\documentclass[10pt,twoside]{article}
\usepackage{amsmath,amssymb,amsthm,mathptmx,hyperref,nicefrac}
\usepackage{titlesec}
\usepackage{indentfirst}
\usepackage[inner=2.4cm,outer=2.65cm,bottom=3cm]{geometry}
\titleformat*{\section}{\normalsize\bfseries}
\titleformat*{\subsection}{\normalsize\itshape}
\usepackage{pstricks}
\usepackage{graphicx}
\usepackage{color}
\definecolor{matheonblue}{RGB}{0,0,50}
\definecolor{matheonlightblue}{RGB}{139,0,0}

\newcommand{\f}[1]{\pmb{#1}}
\DeclareMathOperator{\R}{\mathbb{R}}

\DeclareMathOperator{\C}{\mathcal{C}}
\DeclareMathOperator{\AC}{\mathcal{AC}}

\DeclareMathOperator{\Se}{\mathbb{E}}

\DeclareMathOperator{\N}{\mathbb{N}}
\DeclareMathOperator{\V}{\f H ^1 _{0,\sigma}}
\DeclareMathOperator{\Vd}{(\f H ^{1} _{0,\sigma})^*}

\DeclareMathOperator{\Ha}{\f L^2_{\sigma}}
\DeclareMathOperator{\Hb}{\f{H}^1_0}
\DeclareMathOperator{\Hc}{\f H^2}

\DeclareMathOperator{\He}{\f{H}^1}
\DeclareMathOperator{\Le}{\f{L}^2}

\DeclareMathOperator{\F}{\mathcal{F}}

\renewcommand{\a}{\f a}
\renewcommand{\H}{{H}^2_0}

\newcommand{\Hi}{H^2}
\newcommand{\Hfi}{H^4}
\DeclareMathOperator{\Hz}{\f{H}^2}

\DeclareMathOperator{\ra}{\rightarrow}
\DeclareMathOperator{\de}{\text{d}}
\DeclareMathOperator{\tr}{tr}
\DeclareMathOperator{\con}{konst.}
\DeclareMathOperator{\spa}{span}

\newcommand{\RM}[1]{\MakeUppercase{\romannumeral #1}}

\newcommand{\br}[1]{\frac{\de #1}{\de t}}
\newcommand{\pat}[2]{\frac{\partial #1}{\partial #2}}
\newcommand{\va}[1]{\frac{\delta \mathcal{F}}{\delta #1}}
\DeclareMathOperator{\di}{\nabla \cdot}

\DeclareMathOperator{\sym}{{sym}}
\DeclareMathOperator{\skw}{skw}
\newcommand{\sy}[1]{(\nabla \f {#1})_{{\sym}}}
\newcommand{\sk}[1]{(\nabla \f {#1})_{\skw}}
\renewcommand{\t}{\partial_t  }
\newcommand{\syn}[1]{(\nabla \f {#1}_n)_{\sym}}
\newcommand{\skn}[1]{(\nabla \f {#1}_n)_{\skw}}

\DeclareMathOperator{\curl}{\nabla \times }

\DeclareMathOperator{\clos}{clos}
\newcommand{\inte}[1]{\int_{\Omega}\left({ #1}\right) \de \f x}
\newcommand{\intt}[1]{\int_{0}^T\left({ #1}\right) \de t}
\newcommand{\intte}[1]{\int_{0}^T{ #1} \de t}

\newcommand{\intter}[1]{\int_{0}^T{\left ( #1\right )} \text{\emph{d}} t}
\newcommand{\intet}[1]{\int_{\Omega}{ #1} \de \f x}
\newcommand{\ov}[1]{\overline{#1}}
\newcommand{\ro}[1]{\mathring{{\f #1}}}
\newcommand{\rn}[1]{\mathring{{\f #1}}_n}

\newcommand{\n}{\nabla}

\newcommand{\dn}{\f d_n}
\newcommand{\an}{\a_n}
\renewcommand{\con}{( \f v_n \cdot \nabla)}

\newcommand{\co}{( \f v \cdot \nabla)}

\newtheorem{theorem}{Theorem}
\numberwithin{theorem}{section}
\numberwithin{equation}{section}
\newtheorem{lem}[theorem]{Lemma}
\newtheorem{rem}[theorem]{Remark}
\newtheorem{defi}[theorem]{Definition}
\newtheorem{cor}[theorem]{Corollary}

\begin{document}
\setcounter{footnote}{1}
\author{Etienne Emmrich\thanks{%
        Technische Universit\"{a}t Berlin,
        Institut f\"{u}r Mathematik,
        Stra{\ss}e des 17.~Juni 136,
        10623 Berlin, Germany
        \newline{\tt emmrich@math.tu-berlin.de}
        }
\and    Robert Lasarzik\thanks{%
       	 	Weierstrass Institute,
			Mohrenstr. 39 ,
			10117 Berlin, Germany
		    \newline{\tt  robert.lasarzik@wias-berlin.de}
        }
}%

\title{Existence of weak solutions to a dynamic model for smectic-A liquid crystals under undulations\footnote{Funded by the Deutsche Forschungsgemeinschaft (DFG, German Research Foundation) -- Projektnummer 163436311 -- SFB 910.}
}
\markboth{Weak solutions to a dynamic model for smectic-A liquid crystals}{E.~Emmrich and R.~Lasarzik}
\date{Version \today}
\maketitle
\begin{abstract}
A nonlinear model due to Soddemann et\,al.~\cite{sode} and Stewart~\cite{stewart} describing incompressible smectic-A liquid crystals under flow is studied.
In comparison to previously considered models, this particular model takes into account possible undulations of the layers away from equilibrium, which has been observed in experiments. The emerging decoupling of the director and the layer normal is incorporated by an additional evolution equation for the director.
Global existence of weak solutions to this model is proved via a Galerkin approximation with eigenfunctions of the associated linear differential operators in the three-dimensional case.
\end{abstract}

\tableofcontents
\section{Introduction\label{sec:intro}}
Liquid crystals are materials with remarkable physical and chemical properties.
Displays of electronic devices as those of computers, tablets or smart phones contain and only fulfill their function because of liquid crystals.
As such materials nowadays are an important part of our life, a profound mathematical understanding is more and more necessary. This necessity resulted in many mathematical publications in recent years.
However, different meso-phases received different amounts of attention. While nematic liquid crystals were in the focus, smectic liquid crystals were rarely discussed, even though they are at the core of many applications~\cite{appl}.

In this article, we prove global existence of weak solutions to a model describing smectic-A liquid crystals under flow.
The nonequilibrium behaviour results in undulations of the layers leading to a decoupling of the averaged direction of the molecules, the director, and the layer normal. In comparison to previously considered models (see for instance~\cite{liu}), this decoupling is taken into account by an additional evolution equation for the director.

The existence proof relies on a Galerkin approximation with eigenfunctions of an associated differential operator. To the best knowledge of the authors, the presented result is the first one showing existence of solutions to a model describing smectic-A liquid crystals under flow away from equilibrium.
Before we provide an overview on the existing literature, we give an introduction into the structure of liquid crystals and their different meso-phases.
\subsection{Properties of liquid crystal meso-phases}
As their name already suggests, liquid crystals have properties of solid crystals as well as of  conventional liquids.
On the one hand, these materials consist of rod-like molecules  that form a condensed matter as fluids do. On the other hand, this substances exhibit orientational ordering as solid crystals do. The regime of liquid crystals can be subdivided into different meso-phases depending on positional and orientational ordering. The different meso-phases evolve as a function of temperature (thermotropic liquid crystals) or concentration in a solvent (lyotropic liquid crystals). In the nematic phase,
the rod-like molecules have no positional ordering but are randomly distributed in space (see Figure~\ref{figure}).
\begin{figure}[h]
\begin{center}
\input{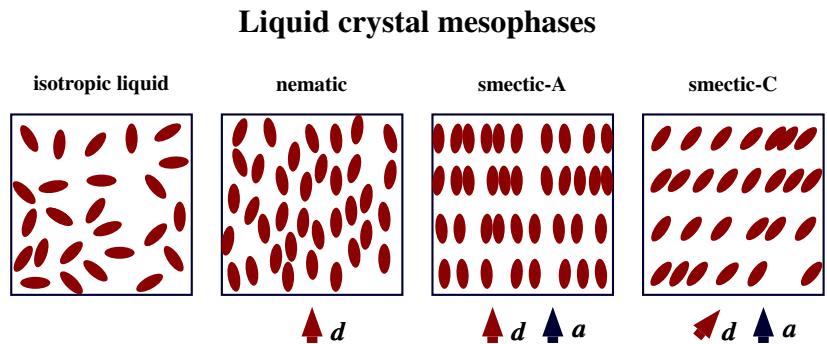}
\end{center}
\caption{Different liquid crystal phases and their molecular structure\label{figure}}
\end{figure}
They tend to align in the same direction, described by the so called director $\f d$, the locally averaged direction of the molecules.
In smectic phases, the molecules are also aligned in the same direction but they obey additional positional ordering. The material density is denoted by $\phi$ and exhibits peaks in one direction.
With other words, the molecules are ordered in layers stacked over each other. The layers can be seen as the isosurfaces of the material density $\phi$ to a certain value.
The normal vector $\a$ of the layers points in the same direction as $\nabla \phi$ since the gradient is always orthogonal to the isosurfaces.
 In the smectic-C phase, there is a fixed angle between the layer normal $\a$ and the director $\f d$, which differs from zero degrees, whereas in smectic-A liquid crystals, the layer normal $\a$ is parallel to the alignment direction $\f d$.
All different phases are illustrated in Figure~\ref{figure}.

This article deals with smectic-A liquid crystals.
\subsection{Review of known results}
In the mathematical community, many articles dealing with the Ericksen--Leslie model have been published. This model describes the nematic phase via a Navier--Stokes-like equation that is nonlinearly coupled with a parabolic equation describing the evolution of the director~$\f d$ (see Figure~\ref{figure}). This model was proposed by Ericksen~\cite{Erick1,Erick2} and Leslie~\cite{leslie,leslie2} and since then extensively studied, see for instance~\cite{unsere,masse,weakstrong,linliu1,linliu3}.

A similar model for smectic-A liquid crystals was proposed by E~\cite{weinan}.
It couples the Navier--Stokes-like equation with a fourth-order partial differential equation modelling the evolution of the layer function~$\phi$. This model assumes that the director~$\f d$ and the layer normal~$\a$ always coincide. The first result on existence of solutions to this model was proved by Liu~\cite{liu} and since then, there have been results proved on the existence~\cite{global}, long-time behaviour~\cite{finite} and numerical approximation~\cite{appro} of the model. A review can be found in~\cite{review}.

In the physical community, it has been observed~\cite{clark,delaye} that layered liquids show a coupling between their internal structure and an applied shear flow. Smectic-A liquid crystals are very sensitive against dilatation of the layers. Above a critical value of the dilatation, the  layers form undulations to diminish the strain locally.
\begin{figure}[h]
\begin{center}
\psset{xunit=0.8cm,yunit=0.8cm}
\begin{pspicture}(0,0)(7.5,5)
\rput(3.25,4.5){\begin{large}\textbf{
Undulation effect in smectic-A liquid crystals}
\end{large}}
\rput(1.5,3.5){
\begin{minipage}{3cm}\begin{center}
\textbf{\footnotesize{equilibrium behaviour}}
\end{center}
\end{minipage}}
\pspolygon[linecolor=matheonblue](0,0)(0,3)(3,3)(3,0)
\pscurve[linecolor=matheonlightblue]{-}(0,0.5)(2,0.5)(3,0.5)
\psline[linecolor=matheonlightblue]{-}(0,1)(3,1)
\psline[linecolor=matheonlightblue]{-}(0,1.5)(3,1.5)
\psline[linecolor=matheonlightblue]{-}(0,2)(3,2)
\psline[linecolor=matheonlightblue]{-}(0,2.5)(3,2.5)

\rput(6,3.5){\begin{minipage}{3cm}\begin{center}
\textbf{\footnotesize{undulation of the layers}}
\end{center}
\end{minipage}}

\pspolygon[linecolor=matheonblue](4.5,0)(4.5,3.0)(7.5,3)(7.5,0)
\pscurve[linecolor=matheonlightblue]{-}(4.5,0.5)(5,0.55)(5.5,0.45)(6,0.55)(6.5,0.45)(7,0.55)(7.5,0.5)
\pscurve[linecolor=matheonlightblue]{-}(4.5,1)(5,1.1)(5.5,0.9)(6,1.1)(6.5,0.9)(7,1.1)(7.5,1)
\pscurve[linecolor=matheonlightblue]{-}(4.5,1.5)(5,1.65)(5.5,1.35)(6,1.65)(6.5,1.35)(7,1.65)(7.5,1.5)
\pscurve[linecolor=matheonlightblue]{-}(4.5,2)(5,2.1)(5.5,1.9)(6,2.1)(6.5,1.9)(7,2.1)(7.5,2)
\pscurve[linecolor=matheonlightblue]{-}(4.5,2.5)(5,2.55)(5.5,2.45)(6,2.55)(6.5,2.45)(7,2.55)(7.5,2.5)

\end{pspicture}
\caption{If a shear flow is applied parallel to the layers, orthorgonal to this sheet of paper, the material reduces the strain by local rotations resulting in undulations with a wave vector orthorgonal to the direction of the shear force.\label{fig2}}
\end{center}
\end{figure}
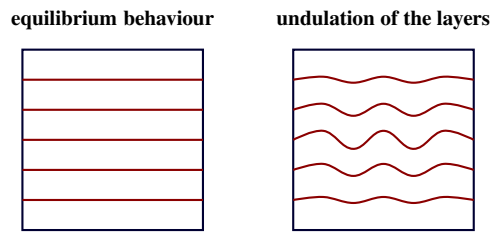
 In the described scenario, the director~$\f d$ may not be parallel to the normal of the layers $\a$ such that $\f d$ and $\a$ may decouple.
This results in the possible effect of permeation, i.e.,\,motion of the fluid through the layers in direction of the layer normal~(see Stewart~\cite{stewart}).

New theories by Auernhammer et\,al.~\cite{auer1,auer2} and Soddemann et\,al.~\cite{sode} include the decoupling of director and layer normal. They propose a system consisting of a Navier--Stokes-like equation coupled with a parabolic equation for the director~$\f d$ and, additionally, with a fourth-order equation for the description of the layers. In a sense, these theories  combine the Ericksen--Leslie model~\cite{leslie} for nematic liquid crystals with the theory by E~\cite{weinan} for smectic-A liquid crystals in equilibrium. This description via three partial differential equations includes possible undulations of the layers and permeation of the fluid through the layers as observed in experiments.
The theory of Auernhammer  and Soddemann   does not impose the gradient of the layer function~$\nabla \phi$ to be of length one as it is done in E~\cite{weinan}.
In contrast to that, Stewart~\cite{stewart} proposed a similar model, where the coupling of the layer function with the other two equations occurs via a normalized gradient of the layer function, i.e.,\,$\a:=\nabla \phi / | \nabla \phi|$. He especially notes that the so-called Oseen constraint $\curl \a=0$, which holds in the equilibrium situation (see De Gennes~\cite[Section~7.2.1.8.]{gennes} and note that $\curl ( \nabla \phi)=0$), does no longer hold. Since the distance of the layers may vary away from equilibrium and since $|\nabla \phi|$ is no longer a constant, the Oseen constraint may be violated, i.e.,\,$\curl \a \neq 0$. In the sequel of this article, we  consider a model that has features of those by Soddemann et\,al.~\cite{sode} and Stewart~\cite{stewart}.
We prove the global existence of weak solutions to the proposed model in the three-dimensional case. To the best knowledge of the authors, this is the first existence result for a nonstationary model describing smectic-A liquid crystals under flow that incorporates possible undulations of the layers, which were observed in experiments.

The paper is organized as follows: In Section~\ref{sec:not}, we introduce some notation. In Section~\ref{sec:sol}, we present the equations of motion and  a precise definition of a solution to the equations considered. Section~\ref{sec:pre} provides certain results that are essential to prove the main Theorem~\ref{thm:main} in Section~\ref{sec:mainpart}.
There, we introduce a Galerkin approximation~(Section~\ref{sec:dis}), which gives rise to a sequence of approximate solutions. For this sequence of solutions, we derive a priori estimates~(Section~\ref{sec:energy}) and show the convergence of a subsequence to the desired solution~(Section~\ref{sec:conv}).
In the last section (see Section~\ref{sec:dicus}), we comment on possible adaptations of the model in regard of the Oseen constraint.

\subsection{Notation\label{sec:not}}
\noindent
We consider a bounded domain $\Omega \in \R^3$ of class $\C^{4}$.
Elements of the vector space $\R^3$ are denoted by bold small letters. Matrices $ \f A \in \R^{3\times 3}$ are denoted by bold capital Latin letters.
In contrast to that, scalar numbers will be denoted by small Latin or Greek letters. Capital Latin letters are reserved for potentials.


The Euclidean inner product in $\R^3$ is denoted by a dot,
$ \f a \cdot \f b : = \f a ^T \f b = \sum_{i=1}^3 \f a_i \f b_i$  for $ \f a, \f b \in \R^3$.
The Frobenius inner product in the space $\R^{3\times 3}$ of matrices is denoted by a colon, $ \f A: \f B:= \tr ( \f A^T \f B)= \sum_{i,j=1}^3 \f A_{ij} \f B_{ij}$ for $\f A , \f B \in \R^{3\times 3}$.
We also employ the corresponding Euclidean norm with $| \f a|^2 = \f a \cdot \f a$ for $ \f a \in \R^3$ and the Frobenius norm with $ |\f A|^2=\f A:\f A$ for $\f A \in \R^{3\times 3}$.
The standard matrix and matrix-vector multiplication, however, is written without an extra sign for bre\-vi\-ty,
$$\f A \f B =\left [ \sum _{j=1}^3 \f A_{ij}\f B_{jk} \right ]_{i,k=1}^3 \,, \quad  \f A \f a = \left [ \sum _{j=1}^3 \f A_{ij}\f a_j \right ]_{i=1}^3\, , \quad  \f A \in \R^{3\times 3},\,\f B \in \R^{3\times3} ,\, \f a \in \R^3 .$$
The outer product is denoted by
$\f a \otimes \f b = \f a \f b^T = \left [ \f a_i  \f b_j\right ]_{i,j=1}^3$ for $\f a , \f b \in \R^3$. Note that
$\tr (\f a \otimes \f b  ) = \f a\cdot \f b$.
The symmetric and skew-symmetric part of a matrix are denoted by $\f A_{\sym}: = \frac{1}{2} (\f A + \f A^T)$ and
$\f A _{\skw} : = \frac{1}{2}( \f A - \f A^T)$ for $\f A \in \R^{3\times  3}$, respectively. For the Frobenius product of two matrices $\f A, \f B \in \R^{3\times 3 }$, we find that
 \begin{align*}
 \f A: \f B = \f A : \f B_{\sym}  \text{ if } \f A^T= \f A\, ,\quad
  \f A: \f B = \f A : \f B_{\skw} \text{ if } \f A^T= -\f A\, .
 \end{align*}
Moreover, there holds $\f A^T\f B : \f C = \f B : \f A \f C$ for
$\f A, \f B, \f C \in \R^{3\times 3}$ as well as
$ \f a\otimes \f b : \f A = \f a \cdot \f A \f b$ for
$\f a, \f b \in \R^3$, $\f A \in \R^{3\times 3 }$. This implies
$ \f a \otimes \f a : \f A = \f a \cdot \f A \f a =  \f a \cdot \f A_{\sym} \f a$.

We use  the Nabla symbol $\nabla $  for real-valued functions $f : \R^3 \to \R$, vector-valued functions $ \f f : \R^3 \to \R^3$ as well as matrix-valued functions $\f A : \R^3 \to \R^{3\times 3}$ denoting
\begin{align*}
\nabla f := \left [ \pat{f}{\f x_i} \right ] _{i=1}^3\, ,\quad
\nabla \f f  := \left [ \pat{\f f _i}{ \f x_j} \right ] _{i,j=1}^3 \, ,\quad
\nabla \f A  := \left [ \pat{\f A _{ij}}{ \f x_k} \right ] _{i,j,k=1}^3\, .
\end{align*}
For brevity, we write $ \nabla \f f^T $ instead of $ ( \nabla \f f)^T$. The symmetric and skew-symmetric part of the gradient of a vector-valued function $\f f$ are denoted by $ \sy f$ and $ \sk f$, respectively.
 The divergence of a vector-valued function $ \f f : \R^3 \to \R^3$ and a matrix-valued function $\f A : \R^3 \to \R^{3\times 3}$ is defined by
\begin{align*}
\di \f f := \sum_{i=1}^3 \pat{\f f _i}{\f x_i} = \tr ( \nabla \f f)\, , \quad  \di \f A := \left [\sum_{j=1}^3 \pat{\f A_{ij}}{\f x_j}\right] _{i=1}^3\, .
\end{align*}
Note that $( \f v\cdot \nabla ) \f f = ( \nabla \f f) \f v = \nabla \f f\, \f v $ for vector-valued functions $ \f v,\,\f f : \R^3 \to \R^3$.
We abbreviate $\nabla \nabla $ by $\nabla^2$.
The double divergence is denoted by $\n^2:$ and defined via
\begin{align*}
\n^2:  \f A &= \sum_{i,j=1}^3 \pat{^2\f A_{ij}}{\f x_i\partial\f x_j}
\end{align*}
for  matrix-valued functions $\f A : \R^3 \to \R^{3\times 3}$. 
The Laplacian is defined as usual by~$ \Delta:= \sum_{i=1}^3 \partial^2_{\f x_i}$ and  the curl by $\curl \cdot$. We abbreviate the bi-Laplacien by $\Delta^2=\Delta\Delta$.

Spaces of vector-valued functions are emphasized by bold letters, for example $\f L^p(\Omega) := L^p(\Omega; \R^3)$.
If it is clear from the context, we also use this bold notation for spaces of matrix-valued functions.
Additionally, the indication of the domain $\Omega$ is often omitted for the brevity of notation.
In the same way, we denote the appropriate Sobolev spaces, i.e.,\,$
\f W^{k,p}(\Omega) := W^{k,p}(\Omega; \R^3)$.
The special Hilbert space cases for $p=2$ are as usual denoted by $ \f H^k(\Omega):= \f W^{k,2}(\Omega)$.
The appropriate spaces for homogeneous Dirichlet boundary conditions are defined as the closure
$ \f H^k_0 (\Omega) = \clos _{\|\cdot\|_{\f H^k(\Omega)}}{\C_{c}^{\infty}(\Omega;\R^3)}   $, where $\C_{c}^{\infty}(\Omega;\R^3)$ denotes the space of infinitely many times differentiable functions with compact support in $\Omega$.

The space of smooth solenoidal functions with compact support is denoted by $\mathcal{C}_{c,\sigma}^\infty(\Omega;\R^3)$. By $\f L^p_{\sigma}( \Omega) $, $\V(\Omega)$,  and $ \f W^{1,p}_{0,\sigma}( \Omega)$, we denote the closure of $\mathcal{C}_{c,\sigma}^\infty(\Omega;\R^3)$ with respect to the norm of $\f L^p(\Omega) $, $ \f H^1( \Omega) $, and $ \f W^{1,p}(\Omega)$, respectively $(1\leq p<\infty)$.

The dual space of a vector space $ V$ is always denoted by  $ V^*$ and equipped with the standard norm;  the duality pairing is denoted by  $\langle\cdot, \cdot \rangle$.
The inner product in $L^2 ( \Omega; \R^3)$  is denoted by $ (\cdot, \cdot )$  and in  $L^2 ( \Omega ; \R^{3\times 3 })$  by $(\cdot ; \cdot )$.
The duality pairing between $\f L^p(\Omega)$ and $ \f L^q(\Omega)$, for conjugated exponents $p$ and $q$, i.e.,\,$1/p+1/q=1$, is also denoted by $(\cdot , \cdot )$ and $( \cdot ; \cdot )$,
 respectively.
 We equip  $\H$ with the norm  $\|\cdot \|_{\H}:=\|\Delta \cdot \|_{L^2}$, which is equivalent to the full~$ H^2$-norm (see \cite[Corollary~2.21]{bilaplace}).
Another important space is $\Hfi\cap \H $ which is equipped with the norm $(\|\Delta^2 \cdot \|_{L^2}^2 + \|\Delta \cdot \|_{L^2}^2)^{1/2}$ (see \cite[Corollary~2.21]{bilaplace} for the equivalence to the standard norm).
The trace operator is denoted by $\f \gamma_0$.

The Bochner spaces for a Banach space $ V$  are as usual denoted by $ L^p(0,T; V)$ ($1 \le p \le \infty$) or $W^{1,s}(0,T; V)$ ($s>0$) for the case that the time derivative is also integrable to the exponent $s$ (see also Diestel and Uhl~\cite[Section II.2]{diestel} or Roub\'i\v{c}ek~\cite[Section 1.5]{roubicek}).
 To abbreviate, we often omit the time interval $(0,T)$ and the domain $\Omega$ and write for example $L^p(\f W^{k,p})$.
By $ \AC ( [0,T]; V)$ and $ \C_w([0,T]; V)$, we denote the spaces of abstract functions mapping $[0,T]$ into $V$,  which are absolutely continuous on $[0,T]$ and continuous on $[0,T]$ with respect to the weak topology on $V$, respectively.

By $c>0$, we denote a generic positive constant.

\section{Model and main result\label{sec:sol}}
In this section, we introduce the system of the equations of motion and state the main result.

We consider the model
\begin{subequations}\label{eq:strong}
\begin{align}
 \ro{ d }+ \lambda \sy v\f d+ 2 \kappa_1\gamma \sy v \a  + \gamma\f q & =0\,,\label{dir}\\
\t  \phi + ( \f v \cdot \nabla ) \phi + \lambda_p j  &= 0 \,\label{lay} ,\\
\t {\f v}  + ( \f v \cdot \nabla ) \f v +  \nabla \pi + \di \f T^E- \di  \f T^V&= \f g\,, \label{nav}\\
\di \f v & = 0\, .
\end{align}%
\end{subequations}
The vector $\f d:\ov{\Omega}\times[0,T]\ra \R^3$ represents the orientation of the rod-like molecules, $\f v : \ov{\Omega}\times [0,T] \ra \R^3$ denotes the velocity  of the fluid and $\phi : \ov{\Omega} \times [0,T] \ra \R$ denotes the layer function.
In this context, $\phi$ is not supposed to resemble the material density but to exhibit the same layers as isosurfaces.
The pressure is denoted by $\pi: \ov \Omega \times [0,T] \ra \R$. We do not consider the existence of the pressure.
The variables $\lambda$, $\kappa_1$, $\gamma$, and $\lambda_p$ are prescribed constants of the system.

The smectic layer normal is usually denoted by $\a$ and is given by the gradient of $\phi$, $\a : = \nabla \phi$.
\begin{rem}\label{rem:a}
With the previous definition of~$\f a$, we follow Soddemann~et\,al.\,~\cite{sode}.
The proof of this paper is also valid for other choices of $\f a$, for example the one introduced in Section~\ref{sec:dicus}. In the proof, we keep the extra variable $\a$ to include other connections between $\nabla \phi$ and $\a$ such that $\f a$ is a continuously differentiable function in $\nabla \phi$. See Section~\ref{sec:dicus} for more details.

Remark that Stewart~\cite{stewart} proposed to take $\f a=\nabla \phi /| \nabla \phi|$. With this choice, $\f a$ is not a continuous function in $\nabla\phi$, which would deprive us of establishing existence of solutions to the approximate system by Carath\'{e}odory's theorem and identifying the limit of these solutions with Lebesgue's theorem on dominated convergence. Hence, the present proof would not work. Nevertheless, we propose a possible relaxation in Section~\ref{sec:dicus}.
\end{rem}

The material derivative of the director is denoted by $\ro  d $ and given by
\begin{align}
\ro{ d } := \t \f d + ( \f v \cdot \nabla ) \f d - \sk v \f d \, .\label{rd}
\end{align}
The free energy potential~$F$ describes the elastic forces in the liquid crystal. It is assumed to depend only on the director $\f d$, the gradient of the layer function $ \phi$ as well as their spatial derivatives, $F= F( \f d, \nabla \f d, \nabla \phi , \nabla ^2 \phi )$.
The free energy functional~$\mathcal{F}$  is then given by
\begin{align*}
\mathcal{F}: \He \times \Hi \ra \R , \quad \mathcal{F}(\f d, \phi):= \int_{\Omega} F( \f d, \nabla \f d, \nabla \phi
, \n^2\phi) \de \f x \,.
\end{align*}

The variational derivative of $\mathcal{F}$ with respect to $\f d$ and $\phi$ is abbreviated by $\f q $ and $j$, respectively (see Furihata and Matsuo~\cite[Section 2.1]{furihata}):
\begin{subequations}\label{vari}
\begin{align}
\f q &:=\frac{\delta \mathcal{F}}{\delta \f d} =  \pat{F}{\f d}(\f d , \nabla\f d, \nabla \phi , \nabla ^2 \phi )-\di \pat{F}{\nabla \f d}(\f d, \nabla \f d, \nabla \phi , \nabla ^2 \phi)\, ,\label{qdef}\\
j&:=\va{\phi}= -\di\pat{F}{\nabla \phi } (\f d , \nabla\f d, \nabla \phi , \nabla ^2 \phi )  +\n^2:  \pat{F}{\nabla ^2 \phi } (\f d , \nabla\f d, \nabla \phi , \nabla ^2 \phi )
\, .\label{jdef}
\end{align}
\end{subequations}
Since $F$ only depends on $\nabla \phi$, we may consider another functional $\widetilde{\mathcal{F}} $, where $\nabla \phi$ is replaced by $\f b$:
\begin{align*}
\widetilde{\mathcal{F}}: \He \times \He \ra \R, \quad \widetilde{\mathcal{F}} ( \f d , \f b ) : = \int_{\Omega} F( \f d, \nabla \f d, \f b
, \n\f b ) \de \f x \,.
\end{align*}
This allows us to compute the variational derivative of $\mathcal{F}$ with respect to $\nabla \phi$.
Since the free energy $\F$ does not depend on the layer function $\phi$ itself, but rather on its spatial derivatives, we can express the variational derivative of $\F$ with respect to $\phi$ via
\begin{align}
j= \va{\phi}= - \di \frac{\delta\widetilde{\mathcal{F}}}{\delta \f b } 
\quad \text{with} \quad
\frac{\delta\widetilde{\mathcal{F}}}{\delta \f b }  &
= \pat{F}{\nabla \phi } (\f d , \nabla\f d, \nabla \phi , \nabla ^2 \phi )  - \di \pat{F}{\nabla ^2 \phi } (\f d , \nabla\f d, \nabla \phi , \nabla ^2 \phi )\, .
\label{Vargradphi}
\end{align}
The stress tensor $ \pi I  +  \f T^E-   \f T^V$ of equation~\eqref{nav} is divided into two parts,  an elastic part~$\pi I + \f T^E$ and a viscous part~$\f T^V$. The elastic part is given by the pressure $\pi I $ and
\begin{subequations}
\begin{align}
\f T^E := \nabla \f d^T \pat{F}{\nabla \f d}  + \nabla \phi \otimes \frac{\delta\widetilde{\mathcal{F}}}{\delta \f b } + \nabla ^2 \phi \pat{F}{\nabla ^2 \phi}\, ,\label{ela}
\end{align}
and the viscous part by
\begin{align}
\begin{split}
\f T^ V &:= \alpha_1 ( \f d \cdot \sy v \f d )  \f d \otimes \f d + \frac{\lambda}{\gamma} ( \f d \otimes \ro d )_{\sym} + \frac{1}{\gamma}( \f d \otimes \ro d )_{\skw}
 + \alpha_4 \sy v +
 2\alpha_5 ( \f d \otimes \sy v \f d ) _{\sym}
 \\ & \quad
 + \frac{\lambda}{\gamma} ( \f d \otimes \sy v \f d )
+ \tau_1 ( \a \cdot \sy v \a ) \a \otimes \a  + 2\tau _2 ( \a \otimes \sy v \a )_{\sym}
\\
& \quad
 + 2\kappa_1  (( \a \otimes \ro d )_{\sym} +( \f d\otimes \sy v \a )_{\skw} ) + 2 \kappa_2 (\f d \cdot \sy v \a ) ( \f d \otimes \a)_{\sym} \\
& \quad+ \kappa_3 \left ( (\f d \cdot \sy v \f d) \a \otimes \a + ( \a \cdot \sy v \a ) \f d\otimes \f d \right ) +  2 \kappa _4 \left ( ( \f d \cdot \sy v \a ) \f d \otimes \f d  + ( \f d \cdot \sy v \f d ) (\f d \otimes \a)_{\sym} \right ) \\
& \quad+ 2 \kappa_5 \left ( ( \f d \cdot \sy v \a ) \a \otimes \a + ( \a\cdot \sy v \a )( \a \otimes \f d)_{\sym}\right ) + 2 \kappa_6 \left ( ( \f d \otimes \sy v \a )_{\sym} + ( \a \otimes \sy v \f d ) _{\sym} \right )\, .
 \end{split}\label{Tv}
\end{align}
\end{subequations}
To guarantee the dissipative character of the system, we assume appropriate restrictions  for the appearing constants
$\lambda_p$, $\gamma$, $\lambda$, $\alpha_i$, $\tau_j$, and $\kappa_k$ with $i\in\{1,4,5\}$, $j\in\{1,2\}$, and $k\in \{1,\ldots,6\}$.
Certain constants need to be positive,
\begin{subequations}
\begin{align}
\lambda_p,\quad \gamma,\quad \alpha_1, \quad\alpha_4,\quad 2\alpha_5+\lambda/\gamma - \lambda^2 / \gamma ,\quad \tau_1 ,\quad \tau_2-2\kappa_1^2 \gamma ,\quad \kappa_2 >0 \, .
\end{align}
Other terms have to be small enough to preserve the dissipative character of the system,
\begin{align}
\begin{split}
4\kappa_3^2<  \alpha_1\tau_1\, ,\quad
8\kappa_4^2  <\alpha_1\kappa_2\, ,\quad
8\kappa_5^2  < \kappa_2\tau_1\, ,\quad
4(\kappa_6-\kappa_1\lambda)^2<   (2\tau_2 -4\kappa_1^2\gamma)( 2\alpha_5+\lambda/\gamma- \lambda^2/\gamma )\, .
\end{split}
\label{con2}
\end{align}
\label{con}
\end{subequations}
We further assume that
$\f g \in L^2(0,T; \Vd)$.
Finally, we impose boundary and initial conditions:
 \begin{subequations}  \label{boundary}
\begin{align}
\f d (  \f x, 0 ) & = \f d_0 ( \f x) \quad&\text{for }& \f x \in \Omega , &
\f d (  \f x ,t ) & = \f d_1 ( \f x ) \quad &\text{for }&(   \f x , t ) \in  \partial \Omega \times [0,T]\, ,\\
\phi(\f x , 0) & = \phi_0(\f x) \quad&\text{for }& \f x \in \Omega ,&
  \nabla \phi (  \f x, t )\cdot \f n( \f x) &= 0= \phi(\f x ,t) \quad &\text{for }&(   \f x , t ) \in  \partial \Omega \times [0,T]\,,\\
\f v(\f x, 0) &= \f v_0 (\f x) \quad&\text{for }& \f x \in \Omega ,&
\f v (  \f x, t ) &= \f 0 \quad &\text{for }&(   \f x , t ) \in  \partial \Omega \times [0,T]
  \, .
\end{align}
\end{subequations}
We always assume that $\f d_1= \f d_0$  on $\partial \Omega$, which is a compatibility condition providing regularity, see Lemma~\ref{lem:bound}.
For the initial conditions, we assume  that $ \f d_0 \in \f H^{1}$ with $\f d_1 \Big |
 _{\partial \Omega } \in  \f H ^{3/2}( \partial \Omega)$, $\phi_0\in \H$, and $\f v_0 \in \Ha$.

We assume homogeneous Dirichlet boundary data for the layer function. However, since  system~\eqref{eq:strong} only depends on derivatives of $\phi$, the system is equally fulfilled if $\phi$ is shifted by a constant.

\subsection{Free energy potential}
For the free energy potential modeling the interaction of molecules and layers, we choose a modified form of the energy introduced by Stewart~\cite{stewart},
\begin{subequations}\label{free}
\begin{align}
 W( \f d , \nabla \f d , \nabla \phi,\nabla^2\phi) &: = \frac{k_1}{2} ( \di \f d ) ^2 + \frac{k_3}{2} | \curl \f d |^2 + \frac{k_5}{2} ( \Delta \phi)^2\label{mitAbl}\\ &\quad  + \frac{B_0}{2} ( | \nabla \phi |^2 +  \f d \cdot \a - 2)^2 + \frac{B_1}{2}  |\f d \times \a|^2\, . \label{ohneAbl}
\end{align}
\end{subequations}
In comparison to the model postulated in Stewart~\cite{stewart}, we added the term $k_3 | \curl \f d|^2$, which was left out because the constant $k_3$ is assumed to be small (see~\cite{parodi}).
Additionally, we take the term $(\Delta \phi)^2$ as given by Auernhammer et\,al.~\cite{auer1}  instead of $(\di \a )^2$, which was proposed by Stewart. In the equilibrium case,  where $| \nabla \phi|  $ is constant, both  formulations coincide.
We do these adjustments since they are essential for our analysis, especially to derive suitable a priori estimates.
In comparison to Stewart~\cite{stewart}, we replaced the term $|\nabla \phi|$ by $|\nabla \phi|^2$ in line~\eqref{ohneAbl}. This ensures that~\eqref{W} remains a continuously differentiable function in $\nabla \phi$.
All constants are assumed to be strictly positive, $ k_1, k_3 , k_5, B_0, B_1>0$.

The terms in line\eqref{mitAbl} model the distortion energy in the liquid crystal, especially the splay and bend deformation of the director and the bending of the smectic layers, respectively. The terms in the second line~\eqref{ohneAbl} represent the coupling between the layers and the director, respectively.

Following a standard relaxation technique, we obtain the free energy by adding penalisation terms to $W$,
\begin{align}
F := W + \frac{1}{4\varepsilon_1}\left ( | \f d|^2 -1 \right )^2 + \frac{1}{4\varepsilon_2}\left ( | \nabla \phi|^2-1\right )^2\, .\label{W}
\end{align}
This relaxation technique allows us to omit the Lagrangian multipliers added to the model by Stewart but, nevertheless, takes into account the algebraic restrictions~$| \f d|=1$ and $| \nabla \phi|=1$.
We consider $\varepsilon_1$ and $\varepsilon_2$ to be small. In this paper, however, we do not consider the limit case $\varepsilon_1\,,\varepsilon_2\,\ra0$.

We refer to~\cite{unsere} for generalized assumptions on the  free energy  in the case of the Ericksen--Leslie model such that the system admits weak solutions.
Additionally, we refer to~\cite{masse,weakstrong} for the singular limit of vanishing penalization, $\varepsilon\ra 0$, in the case of the Ericksen--Leslie system equipped with the Oseen--Frank energy resulting in measure-valued solutions.

\subsection{Existence of weak solutions}
Since the aim of this article is to prove the existence of generalized solutions, we
start with a precise definition of a solution.
Therefore, we derive a reformulation of the elastic stress tensor.
For brevity, we omit the arguments of $F$ and its partial derivatives.
Since the free energy potential $F$ depends   on the four arguments  $\f d$, $\nabla \f d $, $\nabla \phi $, and $\nabla ^2 \phi$,   the spatial derivative of $F$  can be expressed
as
\begin{align}
(\f v\cdot \nabla) F =(\f v\cdot \nabla) F (\f d ,\nabla \f d ,\nabla \phi , \nabla ^2 \phi) =
\pat{F}{\f d} \cdot ( \f v \cdot \nabla ) \f d + \pat{F}{\nabla \f d } : ( \f v \cdot \nabla ) \nabla \f d + \pat{F}{\nabla \phi}\cdot( \f v \cdot \nabla ) \nabla \phi + \pat{F}{\nabla ^2 \phi }: ( \f v \cdot \nabla )\nabla ^2 \phi \, .\label{derF}
\end{align}

In the following calculation, we insert~\eqref{ela}, \eqref{qdef} and~\eqref{Vargradphi}, and differentiate by parts, where the boundary terms vanish since $\f v\in \V$. Using the standard tools of vector analysis (see Section~\ref{sec:not}) yields
\begin{align}
& \left(\f T^E ;\nabla \f v  \right)-\left \langle \nabla \f d ^T \f q ,  \f v\right \rangle - \langle \nabla \phi j, \f v \rangle \notag \\
 &=  \left ( \nabla \f d ^T \pat{F}{\nabla \f d} ; \nabla \f v \right ) + \left ( \nabla \phi  \otimes \frac{\delta\widetilde{\mathcal{F}}}{\delta \f b }; \nabla \f v \right ) + \left ( \nabla ^2\phi \pat{F}{\nabla ^2\phi}; \nabla \f v \right )  - \left ( \nabla \f d^T \left ( \pat{F}{\f d}- \di \pat{F}{\nabla \f d}   \right ) , \f v\right ) + \left ( \nabla \phi \di  \frac{\delta\widetilde{\mathcal{F}}}{\delta \f b }, \f v \right ) \notag  \\
 & = - \left ( ( \f v\cdot \n )\nabla \f d ; \pat{F}{\nabla \f d} \right ) - \left (  \nabla \f d^T \di \pat{F}{\n \f d} , \f v\right )
- \left ( \co \nabla \phi  ,\left (  \pat{F}{\nabla \phi} - \di \pat{F}{\nabla ^2\phi}    \right )\right )- \left (  \nabla \phi  \di \frac{\delta\widetilde{\mathcal{F}}}{\delta \f b }, \f v  \right )  \notag  \\ & \quad  - \left ( \co \n^2 \phi ; \pat{F}{\n^2\phi} \right )- \left (  \n^2 \phi  \di \pat{F}{\n^2 \phi }, \f v \right )  - \left ( \co  \f d, \pat{F}{\f d} \right ) + \left ( \n \f d ^T \di \pat{F}{\n \f d } , \f v \right ) + \left ( \nabla \phi  \di \frac{\delta\widetilde{\mathcal{F}}}{\delta \f b }, \f v\right )
 \notag \\
 &= - \intet{  \left (  ( \f v \cdot \nabla )\f d \cdot \pat{F}{\f d} + ( \f v \cdot \nabla )\nabla \f d : \pat{F}{\nabla \f d} + ( \f v \cdot \nabla )\nabla \phi \cdot \pat{F}{\nabla \phi} + ( \f v \cdot \nabla )\nabla ^2 \phi : \pat{F}{\nabla ^2\phi}\right ) } \\&\quad+ \left ( \nabla ^2 \phi \di \pat{F}{ \n^2\phi} , \f v \right ) - \left ( \nabla ^2 \phi \di \pat{F}{ \n^2\phi} , \f v \right )\notag \\
 &= -\intet{ ( \f v \cdot \nabla ) F } = 0 \, .
\label{identi}
\end{align}
The last equality holds since $\f v$ is solenoidal and the second to the last equality is granted by~\eqref{derF}.
Formula~\eqref{identi} allows us to reformulate equation~\eqref{eq:strong} by incorporating $F$ in a reformulation of the pressure, $
\tilde{\pi}:=\pi+F$, and replacing $\di \f T^E $ by $ -\nabla \f d^T \f q-\nabla \phi j$.

\begin{defi}[Weak solution]\label{defi:weak}
The triple $( \f d ,\phi,\f v )$ is said to be a solution to~\eqref{eq:strong}  if
\begin{align}
\begin{split}
 \f d& \in L^\infty(0,T;\He)\cap  L^2(0,T;\f H^2) \cap W^{1,4/3}  (0,T;   \f  L^{{2}}) \,,
\\ \phi & \in L^\infty(0,T; \H) \cap L^2(0,T; H^4) \cap L^2(0,T; L^2)\,,
\\ \f v &\in L^\infty(0,T;\Ha)\cap  L^2(0,T;\V) \cap W^{1,2}(0,T; (\f H^2 \cap \V)^*)\,,
\end{split}\label{weakreg}
\end{align}
if
\begin{subequations}\label{weak}
\begin{align}
 \intter{( \partial_t \f d + ( \f v \cdot \nabla ) \f d - \sk{v} \f d  + \lambda \sy v\f d   , \f \psi)
+ \gamma \left( 2 \kappa_1  \sy v \a +\f q, \f \psi\right )}
={}&0,
\label{eq:dir}
\\
\intter{\left (\t \phi+ \co \phi, \zeta  \right )  +\lambda_p \left( j,  \zeta \right )}
={}&0, \quad \label{eq:lay}
\\\intter{\left (\partial_t \f v + \co \f v-  \nabla
\f d^T \f q - \nabla \phi  j ,  \f \varphi
\right ) + (\f T^V; \nabla \f \varphi )
- \left \langle \f g ,\f \varphi\right \rangle } ={}&0\quad
\label{eq:velo}
\end{align}%
\end{subequations}
hold for all $ \f \psi \in \mathcal{C}_c^\infty (\Omega \times ( 0,T);\R^3 )$, $\zeta \in \C_c^\infty ( \Omega\times (0,T) )  $, and  solenoidal $ \f \varphi \in \mathcal{C}_c^\infty(\Omega \times  ( 0,T);\R^3)$
and if
the initial conditions
are satisfied as well as $\f \gamma_0( \f d ) = \f d_1$.
\end{defi}
To be precise, our concept of solution is a weak solution concept with respect to the Navier--Stokes-like equation~\eqref{eq:velo} but rather a strong solution concept with respect to the equations~\eqref{eq:dir} and~\eqref{eq:lay} for the inner variables.
Note that the trace operator is denoted by $\f\gamma_0$ (see Section~\ref{sec:not}).
In Corollary~\ref{cor:b}, we prove that all terms of~\eqref{weak} are well-defined under the regularity assumptions of Definition~\ref{defi:weak}. We remark that the initial values are attained in a weak sense since $ \f d \in \C_w ([0,T]; \He) $,  $ \phi \in \C_w([0,T]; \H)$, and  $\f v \in \C_w(0,T; \Ha) $ (compare~\eqref{weakreg} as well as Lions and Magenes~\cite[Ch.~3, Lemma~8.1]{magenes}).
\begin{theorem}[Existence of weak solutions]\label{thm:main}
Let $\Omega$ be a domain of class $\C^{4} $ and
assume~\eqref{con}. For given initial data $( \f d_0,\phi_0, \f v_0) \in  \f H^1 \times \H \times \f L^2_{\sigma}  $, boundary data $\f d_1 \in\f H^{3/2} ( \partial \Omega) $  such that $\f \gamma _0( \f d_0) = \f d_1  $ and right-hand side $\f g \in L^2(0,T; \Vd)$, there  exists a weak solution  to  system~\eqref{eq:strong}--\eqref{W} in the sense of Definition~\ref{defi:weak}.
\end{theorem}
Before we give the proof in Section~\ref{sec:mainpart}, Section~\ref{sec:pre} collects some important inequalities that will be of use later on.

\section{Preliminaries\label{sec:pre}}
\subsection{Important inequalities}
\begin{lem}[Gagliardo--Nirenberg]
\label{lem:nir}
Let the domain $\Omega$ be of class $ \C^{4}$, let  $p\in [2,10/3]  $, $q\in [6,10]$ and let $\theta_1$, $\theta_2 \in [0,2]$ be such that
\begin{align*}
  1 =\frac p 2-\frac{\theta_1}{3}\quad \text{and}\quad    1 = \frac q6 -\frac{\theta_2}{3}\,.
\end{align*}
Then
 there exists a  constant $c>0$ such that the estimates
\begin{align*}
\|\nabla  \f d\|_{L^p(\f L^{p})}  \leq c \|\f d\|_{L^2(\Hz)}^{\theta_1/p} \|\f d\|_{L^\infty(\He)}^{1-\theta_1/p} \,,\quad
\| \f d\|_{L^q(\f L^q)} \leq c \|\f d\|_{L^2(\Hz)}^{\theta_2/q} \|\f d\|_{L^\infty(\He)}^{1-\theta_2/q}
\end{align*}
hold for all functions $\f d\in L^\infty(\He) \cap L^2(\Hz) $.
\end{lem}
See Emmrich and Lasarzik~\cite[Lemma 2.1]{unsere} for the proof of this time dependent version of the Gagliardo--Nirenberg inequality. Note that the Lebesgue exponents in time and space are chosen to be equal for the norm on the left-hand side of the inequality.

\begin{cor}\label{cor:phireg}
Let the domain $\Omega$ be  of class $ \C^{4}$, $k\in \{1,2\}$ and let $p\in [6, 14] $ in the case $k=1$ as well as $ p \in [ 2, 14/3]$ in the case $k=2$ with $\theta\in[0,2]$ fulfilling the relation
\begin{align*}
1\geq \frac{p(2k-1)}{6}- \frac{4\theta}{3}\, , \quad k\in \{ 1,2\}\, .
\end{align*}
Then there exists a constant $c>0$ such that the estimate
\begin{align*}
 \|\phi\|_{L^{p}(W^{k,p})} \leq c \| \phi\|_{L^2(\Hfi)}^{\theta/p}\|\phi\|_{L^\infty(\Hi)}^{1-\theta/p} \,
\end{align*}
is fulfilled for all $\phi \in L^\infty(0,T;\Hi)\cap L^2(0,T;\Hfi)$.
\end{cor}

\begin{cor}\label{cor:vreg}
Let the domain $\Omega$ be of class $ \C^{4}$ and the relation $ 1/p=1/2-2/3r$ be fulfilled for $p \in [2,6]$ and for~$r\in[1,\infty]$. Then there exists a constant $c>0$ such that the estimate
\begin{align*}
 \|\f v\|_{L^r(\f L^{p})} \leq c \| \f v\|_{L^2(\He)}^{2/r}\|\f v\|_{L^\infty(\Le)}^{(r-2)/r} \,
\end{align*}
holds for all $\f v \in L^\infty(0,T;\Ha)\cap  L^2(0,T;\V)$.

\end{cor}

\begin{lem}[Extension operator]\label{lem:bound}
Let $\Omega$ be of class $\C^4$. Then there exists a linear continuous extension operator $\Se : \f H^{3/2}(\partial \Omega) \ra \Hc (\Omega)$ such that $ \f \gamma_0( \Se \f d_ 1) = \f d_1$ and $k_1 \nabla (\di  \Se \f d_1 ) - k_3 \curl \curl\Se \f d_1 =0 $\, for all $ \f d_1 \in \f H^{3/2}( \partial \Omega )$.
Additionally, there exists a constant $c>0$ such that
\begin{align}
 \| \Se \f d_1 \|_{\He }\leq c \| \f d _1 \|_{\f H^{1/2} (\partial\Omega)}  \text{ and}\quad \| \Se \f d_1 \|_{\Hc }\leq c \| \f d _1 \|_{\f H^{3/2}(\partial\Omega)}\,\label{abschrand}
\end{align}
for all $ \f d_1 \in \f H^{3/2}( \partial \Omega )$.
\end{lem}
\begin{proof}
The extension operator is chosen as the solution operator of the problem
\begin{align*}
-(k_1 \nabla (\di  \f h ) - k_3 \curl \curl \f h)= -\di \f \Lambda : \nabla \f h  = 0 \quad \text{in } \Omega\,, \quad \f h = \f d_1 \quad \text{on } \partial \Omega\,.
\end{align*}
Here $\f \Lambda$ is a constant tensor of order four with $\f \Lambda _{ijkl} = k_1 \delta_{ij} \delta_{kl} + k_3 (\delta_{ik}\delta_{jl}- \delta_{il}\delta_{jk})$. This tensor is symmetric, i.e.,\,$\f \Lambda _{ijkl}= \f \Lambda _{klij}$, and strongly elliptic, i.e.,\,$ (\f a \otimes \f b ) : \f \Lambda :(\f a \otimes \f b) \geq \min\{k_1,k_3\} | \f a|^2 | \f b |^2 $ for $k_1 $, $k_3>0$.
There exists a unique solution for every  $ \f d_1 \in \f H^{3/2}( \partial \Omega )$ (see~\cite[Theorem~4.10]{mclean}).
A standard regularity result reveals the estimates~\eqref{abschrand} (see~\cite[Theorem~4.21]{mclean}).
By construction, the image of the associated operator~$\Se$ lies in the kernel of the  operator $-\di \f \Lambda : \nabla$.
\end{proof}
\begin{cor}\label{cor:estim}
There
exists a constant $c>0$ such that the estimates
\begin{align}
\| \f d\|_{\Hc} \leq c \left ( \| \Delta \f d \|_{\Le} + \| \f d_1 \| _{\f H^{3/2}(\partial \Omega)} \right ) \quad \text{and} \quad \| \f d\|_{\He }\leq c \left ( \| \nabla \f d\|_{\Le} + \| \f d_1 \|_{\f H^{1/2}(\partial\Omega}\right )\,
\end{align}
hold for every function $\f d \in \Hc$ with $\f \gamma_0 (\f d) = \f d_1 \in \f H^{3/2}(\partial\Omega)$.
\end{cor}
\begin{proof}
With Lemma~\ref{lem:bound}, we observe that $ \f d - \Se \f d _1 \in \Hc \cap \Hb$.
 For $\partial \Omega \in \C^{4}$, the norm $\| \Delta \cdot\|_{\Le}$ is equivalent to the $\|\cdot \|_{\f H^2}$-norm (see Gilbarg and Trudinger~\cite[Theorem 9.15]{gil}) on $\Hc\cap \Hb $.
The full $\Hc$-norm can be estimated by
\begin{align*}
\| \f d\|_{\Hc} \leq \| \f d - \Se \f d_1 \|_{\Hc} + \| \Se \f d_1 \|_{\Hc} \leq c \| \Delta (  \f d - \Se \f d_1) \|_{\Le} + \| \Se \f d_1 \|_{\Hc}  \leq c\left (\| \Delta \f d\|_{\Le}  + \| \f d_1 \|_{\f H^{3/2}(\partial\Omega}\right )\,.
\end{align*}
 The last inequality follows from~\eqref{abschrand}.
 Similarly, we find with Poincar\'{e}'s inequality
 \begin{align*}
 \| \f d\|_{\He} \leq \| \f d - \Se \f d_1 \|_{\He} + \| \Se \f d_1 \|_{\He} \leq c \| \nabla (  \f d - \Se \f d_1) \|_{\Le} + \| \Se \f d_1 \|_{\He}  \leq c\left (\| \nabla \f d\|_{\Le}  + \| \f d_1 \|_{\f H^{1/2}(\partial\Omega}\right )\,,
 \end{align*}
 where again the last inequality follows from~\eqref{abschrand}.
\end{proof}



\subsection{Free energy}
We present two lemmata capturing important properties of the free energy potential~(\ref{free})--(\ref{W}), which are essential for our analysis. These lemmata provide the well-posedness of Definition~\ref{defi:weak} in Corollary~\ref{cor:b}, the essential a priori estimates in Lemma~\ref{cor:apri2} and the estimates in order to pass to the limit in the nonlinear terms in Lemma~\ref{lem:var}.

\begin{lem}[Coerciveness]\label{lem:coerc}
There exists a possibly large constant $c >0$ and a possibly small constant $\eta>0$ such that the free energy $\F$ and its variational derivatives fulfill the estimates
\begin{align}
\F( \f d , \phi ) &\geq \eta (\|\nabla \f d \|_{\Le}^2 + \|\phi\|_{\H}^2)- c \| \f d_1\|_{\f H^{3/2}(\partial\Omega)}^2 \, ,\label{Fcoerc}\\
 \|\f q \|^2_{L^2(\Le)} & \geq \frac{2\eta}{\gamma}\| \Delta \f d \|_{L^2(\Le)}^2  - c (\| \nabla \f d \| _{L^\infty(\Le)}^6+ \| \phi\|_{L^\infty(\H)}^6+ \| \f d_1\|_{\f H^{3/2}(\partial\Omega)}^6+1)\, ,\label{dcoerc}\\
 \|j\|^2_{L^2(L^2)}& \geq \frac{\eta}{\lambda_p}\|\Delta^2\phi\|_{ L^2 (L^2)}^2 -\frac{\eta}{\lambda_p}\| \Delta \f d \|_{L^2(\Le)}^2 - c (\|\nabla  \f d \| _{L^\infty(\Le)}^{78}+ \| \phi\|_{L^\infty(\H)}^{42}+ \| \f d_1\|_{\f H^{3/2}(\partial\Omega)}^{78}+1)\,\label{phicoerc}
 \end{align}
 for every $\f d\in L^\infty(0,T; \He)\cap L^2(0,T; \Hc)$ with $\f \gamma_0( \f d )=\f d_1 \in \f H^{3/2}(\partial\Omega)$ and $\phi\in L^\infty(0,T;\H)\cap L^2(0,T;\Hfi)$.
\end{lem}

\begin{proof}

Considering the term $\tr( \nabla \f d^2 ) - ( \di \f d)^2 $, we observe with some vector calculus (see Section~\ref{sec:not}) that
\begin{align}
\begin{split}
| \nabla \f d|^2 ={}&  \tr ( \nabla \f d ^T \nabla \f d ) = \tr( \nabla \f d \nabla \f d + 2   \sk d^T \nabla \f d) \\={}& \tr ( \nabla \f d \nabla \f d ) + 2 \tr ( \sk d^T \sk d) =  \tr ( \nabla \f d ^2)+2 | \sk d |^2  \\={}& ( \di\f d )^2 + | \curl \f d |^2 + \tr( \nabla \f d^2 ) - ( \di \f d)^2
 \, .
 \end{split}\label{veccalc}
\end{align}
We used that $ 2 | \sk d |^2 =| \curl \f d |^2$.
The last two terms on the right-hand side of~\eqref{veccalc}  can  be interpreted as the divergence of a vectorfield.  By Gau\ss{}' formula, it is already determined  by the prescribed boundary data,
\begin{align*}
\inte{\tr(\nabla \f d^2) - ( \di \f d)^2 } &= \inte { \di ( \nabla \f d \f d - ( \di \f d ) \f d } = \int_{\partial\Omega} \left (\f n \cdot\nabla \f d \f d - ( \di \f d ) \f n \cdot\f d \right ) \de S \leq c \| \f d \|_{\f H^{3/2}(\partial\Omega)}^2\, .
\end{align*}
Additionally, we recognize that all terms in line~\eqref{ohneAbl} and~\eqref{W} are positive. Thus, they can  be estimated from below by zero.
We define $ k : = \min\{k_1,k_3,k_5\}$ and can estimate the free energy by
\begin{align*}
\F( \f d, \phi ) & \geq k \inte{( \di \f d)^2  + |\curl \f d |^ 2+ \tr( \nabla \f d^2 ) - ( \di \f d)^2 + ( \Delta \phi)^2  - \left (\tr( \nabla \f d^2 ) - ( \di \f d)^2\right ) } \\
&\geq k \inte{| \nabla \f d |^2+ ( \Delta \phi)^2  }
 - k \inte{\tr( \nabla \f d^2 ) - ( \di \f d)^2} \\
& \geq  k \left ( \| \nabla \f d \| _ {\Le}^2 +  \| \phi \| _ {\H} ^2\right  ) - c \| \f d_0\|_{\f H^{3/2} ( \partial \Omega)}^2 \, .
\end{align*}

To prove the two remaining inequalities, we need to calculate the variational derivatives of the free energy explicitly.
The variational derivative $\f q$ is given by
\begin{align}
\begin{split}
\f q = \va{\f d} &= - k_1 \nabla (\di \f d) + k_3 \curl \curl \f d + B_0 ( | \nabla \phi|^2 +  \f d \cdot \a -2)\a + B_1 ( | \a|^2 \f d - ( \f d \cdot \a) \a )+ \frac 1 {\varepsilon_1} ( | \f d|^2 -1 ) \f d \\  &=-k_1 \nabla (\di \f d) + k_3 \curl \curl \f d + R_{\f d} \, .
\end{split}
\label{varFd}
\end{align}
Estimating the $\Le$-norm of the variational derivative $\f q$ while using a consequence of Young's inequality,
\begin{align}
2|\f b_1-\f b_2|^2 \geq |\f b_1|^2 - 2 |\f b_2|^2\,  \quad \text{for all } \f b_1 , \f b _2 \in \R^3\, ,\label{simp}
\end{align}
 gives
\begin{align}
& \left \|\f q \right \| _{\Le}^2 \geq \frac{1}{2}\| k_1\nabla ( \di \f d)- k_3\curl \curl \f d\|_{\Le}^2  -
 \| R_{\f d}
  \|_{\Le}^2
\, .\label{Rd}
\end{align}
In view of~\eqref{varFd} and $\f a = \nabla \phi$, the second part of~\eqref{Rd} can be estimated using H\"older's and Young's inequality,
\begin{align}
\| R_{\f d}
  \|_{\Le}^2\leq{} \int_\Omega \left ( B_0 \left| |\nabla \phi |^3 + |\f d | | \nabla \phi |^2 -2| \nabla \phi|\right| +2  B_1 | \nabla \phi |^2 | \f d| +\frac{1}{\varepsilon_1} \left ( | \f d |^3 + |\f d |\right )  \right ) ^2 \de \f x
  \leq {}  c  ( \|\nabla \phi \|_{\f L^6}^6 + \|\f d \|_{\f L^6}^6 +1) \, .\label{Nummer1}
\end{align}
This first part of~\eqref{Rd} can be calculated using the Hilbert space structure,
\begin{align}
\| k_1\nabla ( \di \f d)- k_3\curl \curl \f d\|_{\Le}^2 = k_1^2 \| \nabla ( \di \f d)\|_{\Le}^2 - 2k_1k_3 ( \nabla ( \di \f d) , \curl \curl \f d ) + k_3 ^2 \|\curl \curl \f d\|_{\Le}^2\, .\label{Younger}
\end{align}
Note that the mixed term $ ( \nabla ( \di \f d), \curl \curl \f d) $ can be estimated by the prescribed boundary conditions, which becomes evident after performing an integration by parts and using that the divergence of the $\curl$-operator is zero,
\begin{multline}\label{intpart}
( \nabla ( \di \f d), \curl \curl \f d) = - \langle \di \f d , \di \curl \curl \f d \rangle + \langle \f \gamma_{\f n} ( \curl \curl \f d ), \f \gamma_0(\di \f d ) \rangle \\\leq \| \curl \curl \f d  \|_{\f H^{-1/2}(\partial\Omega)} \| \di \f d \|_{\f H^{1/2}(\partial \Omega)}\leq c \| \f d \|_{\f H^{3/2}(\partial \Omega)}^2\, .
\end{multline}
The integration-by-parts formula we just used has to be interpreted in a weak sense (compare~\cite[p.~99\,ff.]{mclean}).
The  vector identity $\Delta \f d = \nabla ( \di \f d) - \curl \curl \f d $ leads similarly to~\eqref{simp} to the estimate
\begin{multline}
k_1^2 \| \nabla ( \di \f d)\|_{\Le}^2  + k_3 ^2 \|\curl \curl \f d\|_{\Le}^2 \geq
  \\\frac{1}{2}\min\{k_1^2, k_3^2\} \| \Delta \f d\|_{\Le}^2 + \max\{ k_1^2-k_3^2,0\}  \|\nabla (\di \f d)\|_{\Le}^2 + \max\{k_3^2- k_1^2, 0\} \| \curl \curl \f d\|_{\Le}^2 \, .\label{estimateq}
\end{multline}
Since all terms on the right-hand side of the previous inequality~\eqref{estimateq} are positive, we can estimate the $\Le$-norm of $\f q$ by~\eqref{Rd}, \eqref{Younger}, and~\eqref{intpart} such that
\begin{align*}
 \left \|\f q \right \| _{\Le}^2 &\geq \frac{1}{4}\min\{k_1^2, k_3^2\} \| \Delta \f d\|_{\Le}^2   - c  ( \|\nabla \phi \|_{\f L^6}^6 + \|\f d \|_{\f L^6}^6 +1) -c \|\f d_0 \|_{\f H^{3/2}(\partial\Omega)}^2
\, .
\end{align*}
With the embedding in three dimensions $\He\hookrightarrow \f L^6  $ and Corollary~\ref{cor:estim} the claimed inequality~\eqref{dcoerc} becomes evident.

As a last step, we prove inequality~\eqref{phicoerc}.
The variational derivative $j$ is given by (see~\eqref{jdef},~\eqref{free}, and~\eqref{W})
\begin{align}
j=\va{\phi} &
= k_5 \Delta^2 \phi - B_0\di ( (| \nabla \phi|^2 + \f d \cdot \a - 2 )( 2 \nabla \phi + \f d ) )\notag 
- B_1 \di ( | \f d |^2 \a - ( \f d \cdot \a ) \f d ) - \frac{1}{\varepsilon_2}\di (( | \nabla \phi|^2 - 1 ) \nabla \phi) \\
& = k_5 \Delta^2 \phi - B_0  (2 \nabla \phi + \f d ) \cdot (2 \nabla^2 \phi \nabla \phi + \nabla \f d \nabla \phi + \nabla^2 \phi\f d ) \label{varp} - B_0 ( |\nabla \phi|^2 + \f d \cdot \nabla\phi -2 )(2\Delta\phi + \di\f d) \\
& \quad -B_1( 2 \nabla \phi\nabla \f d ^T \f d + \Delta \phi | \f d|^2 - (( \di \f d)  \f d \cdot \nabla \phi+ \f d \cdot \nabla^2 \phi \f d+ \f d \cdot \nabla \f d \nabla \phi )) \notag - \frac{1}{\varepsilon_2} ( \Delta \phi ( | \n \phi|^2-1) + 2 \nabla \phi\cdot \nabla ^2\phi \nabla \phi )\notag
\\&=: k_5 \Delta^2 \phi +  R_\phi\, \notag.
\end{align}
In the calculation of the variational derivative, we explicitly used the choice~$\f a = \nabla \phi$ (see Remark~\ref{rem:a}).
To estimate this variational derivative from below, we use inequality~\eqref{simp}  such that
\begin{align}
\|j\|_{L^2}^2 \geq \frac{k_5^2}{2}\| \Delta^2 \phi \|_{L^2}^2 - \|R_\phi\|_{L^2}^2 \, .\label{jco}
\end{align}
 In regard of the calculation~\eqref{varp}, we estimate $\|R_\phi\|_{L^2}^2 $ (for $\varepsilon_2$ fixed) by
\begin{align*}
\|R_\phi\|_{L^2}^2& \leq{} c
\left\|   B_0 \left (| \nabla^2 \phi || \nabla \phi |^2 +| \nabla^2 \phi | | \f d |^2+| \nabla^2 \phi |  + | \nabla  \f d| |\nabla\phi|^2 +  | \nabla  \f d ||\f d||\nabla \phi| +  | \nabla  \f d | \right )\right \|_{\Le}^2 \\
&\quad+  c
\left\|  B_1 \left (  | \nabla ^2 \phi | | \f d |^2 + |\nabla \f d | |\f d | |\nabla \phi| \right )+ \frac{1}{\varepsilon_2}\left (| \nabla^2 \phi | (| \nabla \phi |^2 +1)\right )\right \|^2_{\Le}
\\& \leq{} c \inte{| \nabla^2\phi|^2  |\nabla \phi| ^4 + | \nabla^2  \phi|^2 | \f d|^4  + | \nabla^2\phi|^2+ |\nabla \f d|^2 | \nabla \phi|^4+ | \nabla \f d|^2| \f d|^2 | \nabla \phi|^2  + |\nabla \f d|^2  }\, .
\end{align*}
In the following, we
 apply Young's inequality  so  that the norms of the director emerge with certain exponents, i.e.,\,$ 16/5$ in the case of $|\nabla \f d |$ and $48/5$ in the case of $|\f d |$. Other choices for the exponents are possible. Nevertheless, the exponents have to be chosen very carefully so that all terms appearing can be absorbed into the leading order terms, i.e.,\,$\|\Delta\f d\|_{\Le}^2$ and $\|\Delta^2 \phi\|_{\Le}^2$.
Applying Young's inequality in this way yields
\begin{align*}
\|R_\phi\|_{L^2}^2
& \leq
 c \inte{|\nabla \f d|^{16/5} + |  \f d|^{48/5} }  + c\inte{| \nabla ^2 \phi|^2 | \nabla \phi|^4 + | \nabla^2  \phi |^{24/7}+ | \nabla ^2 \phi|^2  +| \nabla \phi|^{32/3} + | \nabla \phi| ^{12}+1 }
\\
& \leq
 c \inte{|\nabla \f d|^{16/5} + |  \f d|^{48/5}  + | \nabla^2  \phi |^{24/7}+ | \nabla \phi| ^{12}+1 }
 \end{align*}
 Thus, the remainder $R_\phi$ is bounded by
\begin{align}
\|R_\phi\|_{L^2(L^2)}^2 \leq {}&  c \left ( \|\nabla \f d\|_{ L^{16/5}(L^{16/5})}^{16/5}  + \| \f d\|_{L^{48/5}(\f L^{48/5})}^{48/5}\right )   + c \left (\|\nabla^2  \phi\|_{L^{24/7}(\f L^{24/7})}^{24/7} + \|\nabla \phi\|_{L^{12}(L^{12})}^{12}+ 1 \right )
\, .\label{Rphi}
\end{align}
The Gagliardo--Nirenberg inequalities of Lemma~\ref{lem:nir} and Corollary~\ref{cor:phireg} yield
\begin{multline}
\|R_\phi\|_{L^2(L^2)}^2 \leq  c \left ( \|  \f d\|_{L^2(\Hc)}^{9/5}\|\f d \|_{L^{\infty}(\He)}^{7/5}  + \|  \f d\|_{L^2(\Hc)}^{9/5}\|\f d \|_{L^\infty(\He)}^{39/5}\right ) \\  + c\left ( \|\Delta^2 \phi\|_{L^2(L^2)}^{15/14} \| \phi\|_{L^\infty(\H)}^{33/14} + \| \Delta^2 \phi\|_{L^2(L^2)}^{3/2} \| \phi\|_{L^\infty(\H)}^{21/2}+ 1 \right )\, .\label{AbsR}
\end{multline}
These terms are estimated again with Young's inequality. We apply it  so that the leading order terms, i.e.,\,$\|\Delta\f d \|_{\Le}$ and $\|\Delta^2\phi \|_{L^2}$, remain squared and multiplied with a small constant. This small constant is defined by
\begin{align}
 \eta := \min\left \{  \frac{1}{8}\gamma \min\{k_1^2, k_3^2\},\frac{1}{4} \lambda_p k_5^2, \min\{k_1,k_3,k_5\}\right \}\, .
\end{align}
With this newly defined parameter, we estimate~\eqref{AbsR} further on with Corollary~\ref{cor:estim} and Young's inequality
\begin{align}
\begin{split}
\|R_\phi\|_{L^2(L^2)}^2 \leq{}& \frac{\eta }{\lambda_p} \left ( \| \Delta \f d\|_{L^2(\Le)}^{2}+ \|\Delta^2 \phi\|_{L^2(L^2)}^{2}+ \| \f d_1 \|_{\f H^{3/2}(\partial\Omega)}^ {2}  \right )
\\&+ c \left (
\| \f d \|_{L^\infty(\He)}^{14}  + \|\f d \|_{L^\infty(\He)}^{78}
+\| \phi\|_{L^\infty(\H)}^{66/13}  +  \| \phi\|_{L^\infty(\H)}^{42}+ 1 \right )
\\
 \leq{}& \frac{\eta}{\lambda_p}  \left ( \|\Delta \f d\|_{L^2(\Le)}^{2}+ \|\Delta^2 \phi\|_{L^2(L^2)}^{2}\right )
+ c \left (
 \|\nabla \f d \|_{L^\infty(\Le)}^{78}   +  \| \phi\|_{L^\infty(\H)}^{42}+ \| \f d_1 \|_{\f H^{3/2}(\partial\Omega)}^ {78}+ 1 \right )
\, .
\end{split}
\label{Rphiest}
\end{align}
This estimate can be inserted into~\eqref{jco}, which gives the coercivity estimate claimed for the variational derivative of $\F$ with respect to $\phi$ (see~\eqref{phicoerc}).

 \end{proof}

\begin{lem}[Boundedness]\label{lem:Fdis}
There exists a constant $c>0$ such that the free energy $\F$ (see~\eqref{free} and~\eqref{W}) and its variational derivatives~\eqref{vari} can be estimated by
\begin{align}
\F( \f d_0 ,  \phi_0 ) &\leq c ( \|\f d_0 \|_{\He}^4+ \| \phi _0 \|_{\Hi}^4+1)\, , \label{Fn}\\
\left \|\f q \right \|^2 _{L^2(\Le)} &\leq c ( \|\f d \|^2_{L^2 (\Hz)}  + \| \f d \|_{L^\infty ( \He)}^6+ \| \phi\|_{L^\infty ( \Hi)}^6)\, ,\\
\|j\|^2_{L^2(L^2)}  & \leq  c( \| \phi \|^2_{L^2(\Hfi)} +  \| \f d\|^2_{L^2(\Hz)}+ \|\f d \|_{L^{\infty}(\He)}^{78} +  \| \phi\|_{L^{\infty}(\Hi)}^{42} +1 ) 
\end{align}
 for every
$ \f d_0 \in \He $,
  $\f d\in L^\infty(0,T; \He)\cap L^2(0,T; \Hc)$ fulfilling $\f \gamma_0( \f d_0 )=\f \gamma_0( \f d )=\f d_1 \in \f H^{3/2}(\partial\Omega)$ as well as  $ \phi_0\in \H$,  $\phi\in L^\infty(0,T;\H)\cap L^2(0,T;\Hfi)$.

\end{lem}
\begin{proof}
The free energy is estimated by inserting the definitions~\eqref{free} and~\eqref{W}:
\begin{align*}
\F( \f d_0 ,  \phi_0)
& \leq c ( \| \f d_0\|_{\He}^2 + \| \f d_0\|_ {\f L^4}^4 + \| \phi_0 \|_{\H}^2 + \|  \phi_0\|_{W^{1,4}}^4 +1)\, .
\end{align*}
The continuous embeddings $\He \hookrightarrow \f L^4  $ and $\Hi\hookrightarrow W^{1,4}
$  yield
\begin{align}
\F(  \f d_0 ,  \phi_0 ) \leq c ( \|\f d_0 \|_{\He}^4+ \| \phi _0 \|_{\Hi}^4+1)
\end{align}

In the same way, we estimate both variational derivatives from above.
Using~\eqref{varFd} and~\eqref{Nummer1}, we bound $\f q$ by
\begin{align*}
 \left \|\f q \right \|^2 _{L^2(\Le)} &\leq 2\| k_1\nabla ( \di \f d)- k_3\curl \curl \f d\|^2_{L^2(\Le)}  +
2 \| R_{\f d}
  \|^2_{L^2(\Le)}  \leq c ( \|\f d \|^2_{L^2 (\Hz)}  + \| \f d \|_{L^\infty ( \He)}^6+ \| \phi\|_{L^\infty ( \Hi)}^6+1)\,
\end{align*}
and with ~\eqref{varp} and~\eqref{Rphiest} the variational derivative~$j$ by
\begin{align*}
\|j\|^2_{L^2(\Le)}  & \leq 2 k_5 \|\Delta^2 \phi\|^2_{L^2(L^2)} +  2 \|R_\phi \|^2_{L^2(L^2)}\leq c( \| \phi \|^2_{L^2(\Hfi)} +  \| \f d\|^2_{L^2(\Hz)}+ \|\f d \|_{L^{\infty}(\He)}^{78} +  \| \phi\|_{L^{\infty}(\Hi)}^{42} +1 ) \, .
\end{align*}

\end{proof}

The following corollary proves that our notion of a solution~\eqref{defi:weak} makes sense.
\begin{cor}\label{cor:b}
All terms in~\eqref{weak} are well-defined.
\end{cor}
\begin{proof}
In regard of the assumed regularity~\eqref{weakreg}  for  $\f d$, $\phi$, and $\f v$ as well as the boundedness of the variational derivatives~(Lemma~\ref{lem:bound}), every term appearing in the equations~\eqref{eq:dir} and~\eqref{eq:lay}
is finite, which follows from H\"older's inequality.

The Navier--Stokes-like equation can be considered in a similar fashion. However, in order to handle the viscous stress tensor~$\f T^V$~(see~\eqref{Tv}), the material derivative~$\ro d$ given in~\eqref{rd} needs to be bounded in an appropriate norm. This can be observed by
\begin{align}
\|\ro d \|_{L^{4/3}(\f L^{2})}\leq \| \t \f d \|_{L^{4/3}(\f L^{2})} + \| \f v \|_{L^2(\f L^6)} \| \f d\|_{L^4 ( \f W^{1,3})} + \|\f v \| _{ L^2 ( \He)}\|\f d\|_{L^4( \f L^\infty)}\, .\label{brd}
\end{align}
\end{proof}

\section{Galerkin approximation and proof of the main result\label{sec:mainpart}}
In this section, we prove the main result Theorem~\ref{thm:main} via convergence of a Galerkin approximation. The proof is divided into the following steps:
We first introduce the Galerkin scheme and deduce the local existence of a solution to the approximate problem~(see p.~\pageref{sec:dis}). Then, we derive a priori estimates~(Section~\ref{sec:energy}) and we show that the solutions of the approximate problems exist on the whole time interval~$[0,T]$.

 The crucial part is carried out in Section~\ref{sec:conv}, in which we
use the a priori estimates to extract a weakly convergent subsequence. We conclude the convergence of the time derivatives which allows us to  deduce strong convergence~(see Lemma~\ref{lem:wkonv}). This enables us to prove the convergence for the nonlinear variational derivative of the free energy~(see Lemma~\ref{lem:var}) and, therewith, we can pass to the limit in the director equation, the layer equation, and the Navier--Stokes-like equation to obtain in the limit the weak formulation in the sense of Definition~\ref{defi:weak}.


%
\subsection{Galerkin approximation and local existence\label{sec:dis}}
We are going to use a Galerkin scheme to discretize the system of interest in space.
For the approximation of the director equation, we use an $\Le$-orthonormal Galerkin basis consisting of eigenfunctions $\f y_1 , \, \f y_2 , \, \ldots$
of the differential operator corresponding to the boundary value problem
\begin{align}\label{boundaryvalueproblem}
\begin{split}
- k_1 \nabla (\di \f y) + k_3 \curl \curl \f y = - \di \left( \f \Lambda : \nabla \f y \right) &= \f h \quad\text{in } \Omega \, , \\
\f z &= 0\quad\text{on } \partial\Omega \,.
\end{split}
\end{align}
Here, the constant symmetric strongly elliptic tensor~$\f \Lambda$ is defined in the proof of Lemma~\ref{lem:bound}.
The above problem  is a symmetric strongly elliptic system that possesses a unique weak solution $\f z \in \f H^1_0$ for any $\f h \in \f H^{-1}$ (see, e.g., Chipot~\cite[Thm.~13.3]{chipot}). Its solution operator is thus a compact selfadjoint operator in $\f L^2$. Hence there exists an orthogonal basis of eigenfunctions  $\f y_1 , \, \f y_2 , \, \ldots$ in $\Le$.
A regularity result (see McLean~\cite[Theorem~4.21]{mclean}) provides regularity of the eigenfunctions such that $ Y_n : = \spa \left \{  \f y_1, \dots , \f y_n \right \} \subset \f H^2 \cap \Hb $. The associated orthogonal $\Le$-projection is denoted by $R_n
: \Le \longrightarrow Y_n$.
Note that the projection~$R_n$ is $\Hb$-stable, i.e.,\,there exists a constant $c>0$ such that $\| R_n \f y \|_{\Hb}\leq c \|\f y \|_{\Hb} $ for all $\f y \in \Hb$ (see~\cite[Section~4.1]{unsere}).

For the approximation of the layer equation, we consider a Galerkin basis consisting of eigenfunctions of the biharmonic operator.
Consider the boundary value problem
\begin{align*}
  \Delta^2 z_n =  h, \quad \text{in  } \Omega\,,\quad \f n \cdot \nabla z_n = z_n = 0 \quad \text{on }\partial\Omega\, .
\end{align*}
This boundary value problem possesses a unique weak solution for every $ h \in ({\H})^*$ and the solution operator is  a compact selfadjoint operator as a mapping in $L^2$.
Thus, we can find a sequence of eigenfunctions $\{ z_n\}$ that are orthonormal in $L^2$ (see~\cite{biharmonic}).
A standard regularity result shows~$z_n \in \Hfi \cap \H$ for all $n \in \N$ (see~\cite[Corollary~2.21]{bilaplace}).
Then the approximation space is denoted by $Z_n : =\spa \{ z_1,\dots , z_n\}$ with the associated orthogonal projection $Q_n: L^2 \ra Z_n$.
Note that the projection~$Q_n$ is $ \H $-stable, i.e.,\,there exists a constant $c>0$ such that $\| Q_n z \|_{\H}\leq c \|z \|_{\H} $ for all $z \in \H$ (see~\cite[Section~9.8]{brezisbook}). 
The high regularity $\partial \Omega\in\C^4$ of the considered domain is essential to apply the regularity result. This regularity is indeed the only reason, why we have to choose such a regular domain.

For the approximation of the Navier--Stokes-like equation, we follow
Temam~\cite[p.~27f.]{temam} and use a Galerkin basis consisting of  eigenfunctions $\f w_1, \, \f w_2 , \, \ldots \in \f H^2\cap \V $ of the Stokes operator (with homogeneous Dirichlet boundary conditions). As is well known, the eigenfunctions form an orthogonal basis in $\Ha$ as well as in $\V$ and in $\f H^2\cap \V $. Let $W_n= \spa \left \{ \f w_1, \dots , \f w _n\right \} $ ($n\in \N$)
and let $P_n : \Ha \longrightarrow W_n$ denote the
$\Ha$-orthogonal projection onto $W_n$.
Since $\Omega$ is of class $\C^4$, there exists  $c>0$ such that $\|P_n \f v\|_{\f H^2} \le c \|\f v\|_{\f H^2}$ for all $n\in \N$ and $\f v \in \f H^2 \cap \V$, 
see, e.g.,  M\'{a}lek et al.~\cite[Appendix, Thm.~4.11 and Lemma~4.26]{malek} together with Boyer and Fabrie~\cite[Prop.~III.3.17]{boyer}.


The approximate
 problem is the following: Find a solution  $(\f d_n,\phi_n, \f v_n)$ with  $(\f d_n- \Se\f d_1) \in  \AC([0,T]) \otimes  Y_n$, $ \phi_ n \in \AC([0,T])\otimes Z_n$, and  $\f v_n\in  \AC ([0,T]) \otimes W_n      $ solving the problem
\begin{subequations}\label{eq:dis}
\begin{align}
(\rn d  + \lambda  \syn v \f d_n + 2\kappa_1\gamma \syn v \an   +\gamma\f q_n, \f y ) &=0,
& \f d_n(0)&=  \Se \f d_1 + R_n(\f d_0- \Se \f d_1)\,,
\label{ddis}\\
( \partial_t \phi_n  + ( \f v_n \cdot \nabla ) \phi_n+ \lambda_p  j_n , z )  &=0,
& \phi_n(0)&= Q_n\phi_0\,,
\label{phidis}\\
( \partial_t {\f v_n}+ ( \f v_n \cdot \nabla ) \f v _n- \nabla \f d_n ^ T  \f q_n  - \nabla \phi_n j_n, \f w  )+ \left (\f T^V_n; \nabla \f w \right )&= \left \langle \f g,    \f w\right \rangle,&
\f v_n(0) &= P_n \f v_0\,
\label{vdis}
\end{align}
for all  $ \f y \in Y_n$, $  z \in Z_n$, $ \f w \in W_n$, and for a possibly short time interval $[0,T_n)$.
In the above equations, the variational derivatives
and the material derivative are given by
 \begin{align}
\f q_n &: = R_n \left ( \frac{\delta \mathcal{F}}{\delta \f d }( \f d_n, \phi_n )\right ) = R_n \left (\pat{F}{\f d}( \f d_n,\nabla \f d_n ,\nabla \phi _n,\nabla^2\phi_n ) - \di\pat{F}{\nabla \f d}( \f d_n,\nabla \f d_n ,\nabla \phi _n,\nabla^2\phi_n)\right )\, ,\label{qn}\\
j_n &: = Q_n \left ( \va{\phi} ( \f d_n, \phi_n )\right ) = Q_n \left ( \nabla ^2 : \pat{F}{\nabla ^2 \phi }( \f d_n,\nabla \f d_n ,\nabla \phi _n,\nabla^2\phi_n) -\di \pat{F}{\nabla\phi}( \f d_n,\nabla \f d_n ,\nabla \phi _n,\nabla^2\phi_n)\right )\, ,\\
\rn d &:= \t \f d_n + ( \f v_n \cdot \nabla ) \f d_n - \skn v \f d_n \, ,
\end{align}
and the viscous stress by
\begin{align}
\begin{split}
\f T^V_n& :=
\alpha_1 (\f d_n \cdot  \syn v \f d_n )\f d_n \otimes \f d_n
+ \left (  2\alpha_5 + \frac{\lambda}{\gamma} - \frac {\lambda^2}\gamma \right ) ( \f d_n \otimes \syn v \f d_n )_{\sym} \\ & \quad + 2\left (  \kappa_6 - \lambda\kappa_1\right )\left ( ( \syn v \an \otimes \f d_n ) _{\sym}  + ( \syn v \f d_n \otimes \an ) _{ \sym } \right ) - \lambda ( \f q_n \otimes \f d _n ) _ {\sym} + ( \f q_n \otimes \f d_n ) _{\skw}
\\ & \quad+\alpha_4 \syn v
+ \tau_1 ( \an \cdot \syn v \an )( \an \otimes \an ) + ( 2\tau_2 -4 \kappa_1^2 \gamma) ( \an \otimes \syn v \an )_{\sym}  - 2 \kappa_1 \gamma ( \an \otimes \f q_n ) _{\sym} \\ & \quad + 2 \kappa_2 ( \f d_n \cdot \syn v \an )(\f d_n \otimes \an )_{\sym} + \kappa_3 \left (( \dn \cdot \syn v \f d_n )(\an \otimes \an) + ( \an\cdot \syn v \an )( \f d_n \otimes \f d_n )\right ) \\
& \quad + 2\kappa_4 \left (\left ( \f d_n\cdot \syn v \an \right ) ( \f d_n \otimes \f d_n ) + ( \f d_n \cdot \syn v \f d_n ) ( \f d_n \otimes \an ) _{\sym}\right )\\ & \quad + 2 \kappa_5 \left (( \f d_n \cdot \syn v \an ) ( \an \otimes\an) +  ( \an \cdot \syn v \an ) ( \f d_n \otimes \an )_{\sym}\right )\, ,
\end{split}\label{lesliedis}
\end{align}%
\end{subequations}%
which follows from~\eqref{Tv} by inserting~\eqref{ddis} for $ \rn d$. This definition avoids that the time derivative of $\dn$ appears in the viscous stress.

Note that the approximate variational derivative~\eqref{qn} is the variational derivative of the approximate free energy function.
Indeed, defining $\F_n(\f d, \phi) := \F(R_n \f d, Q_n\phi) $, we find that
\begin{align}
\left \langle\frac{\delta \F_n}{\delta \f d}( \f d, \phi), \f \psi \right \rangle = \left \langle R_n \frac{\delta \F}{\delta \f d }( R_n \f d,Q_n \phi ), \f \psi\right \rangle \quad \text{for all } \f \psi \in \Le \, .
\end{align}
The same holds true for the variational derivative of $\F$ with respect to~$\phi$.
Note that the nonlinear terms depending on $\f d$ or $\phi$ and appearing in the Navier--Stokes-like equation~\eqref{nav} have
been projected appropriately.
 This ensures that the important energy inequality is  valid in the approximate setting.

 A classical existence theorem (see Hale~\cite[Chapter I, Theorem 5.2]{hale}) provides, for every $n\in\N$, the existence of a maximal solution to the above approximate problem~\eqref{eq:dis} on an
interval~$[0,T_n)$ in the sense of Carath\'e{}odory.
 This theorem grants a solution  on $[0,T]$ if the solution undergoes no blow-up. With the a priori estimates of the next section, we can exclude blow-ups and thus prove global-in-time existence. \label{sec:exloc}


\subsection{Energy inequality and a priori estimates\label{sec:energy}}
An essential tool for the analysis in this  work is the energy inequality proved in the following.

\begin{lem}
\label{lem:1}
Let the assumptions of Theorem~\ref{thm:main} be fulfilled and let $( \f d_n, \phi_n, \f v_n ) $ be a  solution to \eqref{eq:dis}. Then the energy equality
\begin{align} \begin{split}
  & \frac{1}{2}\|\f v_n(t)\|_{\Le}^2 +  \F(\f d_n(t), \phi_n(t))  +\int_0^t\left ( \gamma  \| \f q_n \|^2_{\Le}+ \lambda_p \left \Vert  j_n\right \Vert _{L^2}^2
 + \alpha_1\Vert \dn \cdot \syn v \dn \Vert _{L^2}^2 + \alpha_4 \| \syn v \|_{\Le}^2\right ) \text{\emph{d}} s  \\
 & \quad +\int_0^t\left ( \left ( 2 \alpha_5 + \frac{\lambda}{\gamma} - \frac {\lambda^2}\gamma \right ) \|\syn v \dn \|^2_{\Le} + \tau_1 \| \an\cdot \syn v \an\|^2_{L^2} \right )\text{\emph{d}} s
 \\ & \quad +\int_0^t \left (( 2\tau_2 -4 \kappa_1^2 \gamma) \| \syn v \an \|^2_{\Le} + 2 \kappa_2 \| \dn \cdot \syn v\an \|_{L^2}^2\right )\text{\emph{d}} s
 \\
&= \frac{1}{2}\|\f v_n(0)\|_{\Le}^2 +  \F(\f d_n(0), \phi_n(0)) + \int_0^t \langle \f g, \f v_n\rangle \text{\emph{d}} s \\ & \quad - 2 \kappa_3\int_0^t  ( \dn\cdot\syn v \dn,\an\cdot\syn v\an )\de s   - 4 \kappa_4 \int_0^t(\dn \cdot \syn v \an,\dn \cdot \syn v \dn )\text{\emph{d}} s  \\ & \quad- 4 \kappa_5 \int_0^t(  \dn \cdot\syn v \an , \an \cdot \syn v \an )\de s  - 4(  \kappa _6 - \kappa_1 \lambda) \int_0^t  ( \syn v \dn ,\syn v \an)\text{\emph{d}} s 
\end{split}\label{entro1}
\end{align}
holds for all $t$ in any compact subinterval of $[0,T_n)$.
\end{lem}
\begin{proof}
In order to derive~\eqref{entro1}, we test the  Navier--Stokes-like equation~\eqref{vdis} with the approximate solutions $\f v_n $ of the velocity field and obtain
\begin{align}
\frac{1}{2}\br{} \| \f v_n \|_{\Le}^2   - \langle \nabla \f d_n^T \f q_n + \nabla\phi _n  j_n , \f v_n \rangle  + ( \f T_n^V;\nabla \f v_n) =  \langle \f g , \f v_n\rangle \, .\label{testvn}
\end{align}
Here, we employed that the convection term  vanishes since $\f v_n$ is solenoidal.
The director equation~\eqref{ddis} is tested with the variational derivative $\f q_n$ (see~\eqref{qn}),
\begin{align}
    (  \t \dn, \f q_n) + (  \con \dn, \f q_n)- (  \skn v \dn, \f q_n)
+\lambda ( \syn v \f d_n , \f q_n) + 2\kappa_1\gamma ( \syn v \an , \f q_n) + \gamma \| \f q_n\|_{\Le}^2
    =0 \, .
\label{testdn}
\end{align}
Note that the projection $R_n$ is well-defined as a mapping $ R_n : \Le \ra Y_n$. This assures that the test function $\f q_n$ is in the appropriate test space.
The layer equation is tested with the approximate variational derivative $j_n$,
\begin{align}
\left ( \t \phi_n , j_n \right )+ \left ( \con \phi_n ,  j_n \right ) + \lambda_p \left \|  j_n\right \|_{L^2}^2 = 0 \, .
\label{testphin}
\end{align}
Note that $Q_n$ is  the $L^2$-projection onto $Z_n$ and $j_n$ is an appropriate test function.

The derivative with respect to time of the free energy  is given by (compare with~\eqref{derF})
\begin{align}
\br{} \mathcal{F}(\f d_n, \phi_n) &= \intet{ \partial_t F( \f d_n, \nabla \f d_n, \nabla \phi _n, \n^2\phi _n) }\notag\\
&=  \inte{ \pat{F}{\f d} \cdot \t \f d_n  + \pat{F}{\n \f d} : \t \nabla \f d_n + \pat{F}{\nabla \phi} \cdot \t \nabla \phi_n + \pat{F}{\n^2\phi}: \t \nabla ^2 \phi_n } \notag\\
&=\inte{ \t \f d_n\cdot \left( \pat{F}{\f h}  - \di \  \pat{F}{\n \f d}\right) + \t \phi_n \left ( \n^2: \pat{F}{\n^2\phi }- \di \pat{F}{\nabla \phi} \right )  }\notag\\
& = \inte{ \t \f d_n \cdot R_n\left( \pat{F}{\f h}  - \di \  \pat{F}{\n \f d}\right)  + \t \phi_n Q_n  \left ( \n^2: \pat{F}{\n^2\phi }- \di \pat{F}{\nabla \phi} \right )   }
\notag\\ &  = ( \t \f d_n , \f q_n) + ( \t \phi_n , j_n)\, .
\label{Fd3}
\end{align}
The boundary terms arising due to the integration by parts formula are zero since the boundary value prescribed for the director is constant in time.
Moreover, $\partial_t \phi_n(t) \in \H$. The second to the last equality in the above calculation is valid since $\partial_t \f d_n (t) \in Y_n$ and $\partial_t \phi_n(t) \in Z_n$.



Summing up the equations~\eqref{testvn},~\eqref{testdn}, and~\eqref{testphin}  while simultaneously using~\eqref{Fd3} leads to
\begin{subequations}
\begin{align}
\br{}& \Big (\frac{1}{2}\|\f v_n\|_{\Le}^2+  \mathcal{F}(\f d_n,\phi_n) \Big )
+\gamma\| \f q_n\|^2_{\Le}+ \lambda_p \left \|  j_n\right \|_{L^2}^2 \label{go}
 \\& - \langle \nabla \f d_n^T \f q_n, \f v_n \rangle- \langle \nabla \phi_n  j_n , \f v_n \rangle+ (  \con \dn,\f q_n) +
 \left ( \con \phi_n ,  j_n \right ) \label{null}
 \\ & + \lambda (\syn v\f d_n, \f q_n)+ 2 \kappa_1\gamma ( \syn
  v \an , \f q_n )  - (  \skn v \dn,\f q_n )\label{last}
  \\   &+ ( \f T_n^V; \nabla \f v_n)
= \langle \f g, \f v_n \rangle\,. \notag%
\end{align}\label{add}%
\end{subequations}%
Line~\eqref{null} vanishes (see also Section~\ref{sec:not}).
 We calculate the last term on the left-hand side, i.e.,\,the viscous stress tested with the gradient of the approximate solution $\f v_n$:
\begin{subequations}\label{les}
\begin{align}
\begin{split}
&(\f T_n^ V; \nabla \f v_n )  :=  \alpha_1 \| \dn \cdot \syn v \dn \|_{L^2}^2  + \alpha_4 \|\syn v\|_{\Le}^2
+\left (  2\alpha_5 + \frac{\lambda}{\gamma} - \frac {\lambda^2}\gamma \right )\| \syn v \dn \|_{\Le}^2
 \\ &\quad + \tau_1 \| \an\cdot  \syn v \an \|^2_{L^2}
 + ( 2\tau_2 -4 \kappa_1^2 \gamma)\| \syn v \an \|_{\Le}^2
+ 2 \kappa_2 \|\dn \cdot \syn v \an \|_{L^2} ^2
\end{split}\label{good}\\ & \quad
- \lambda ( \syn v \f d_n ,\f q_n  )  + ( \f q_n, \skn v \f d_ n )
- 2 \kappa_1 \gamma (\syn v  \an  ,\f q_n )
 \label{last2}
\\
& \quad + 2\kappa_3  (\dn \cdot \syn v \dn, \an \cdot \syn v \an )  \quad +  4 \kappa _4  ( \dn \cdot \syn v \an , \dn \cdot \syn v \dn )  \notag \\
& \quad+ 4 \kappa_5 ( \dn \cdot \syn v \an , \an\cdot \syn v \an \notag) + 4( \kappa_6-\lambda\kappa_1)  ( \syn v \dn , \syn v \an )\,.\notag
\end{align}
\end{subequations}
The sum of the terms in line~\eqref{last2} and the terms in line~\eqref{last} is zero.

Inserting the last equation~\eqref{les} into~\eqref{add}, putting the terms which are not necessarily of positive sign on the right-hand side, and integrating in time yields the energy identity~\eqref{entro1}.
\end{proof}

\begin{cor}
\label{cor1}
Let  the  assumptions of Theorem~\ref{thm:main} be fulfilled.
  Then there exist  positive  constants $\beta_i>0$, $i \in \{ 1, \ldots , 6\}$, and $c>0$ such that
\begin{multline}
   \frac{1}{2}\|\f v_n(t)\|_{\Le}^2 +  \F(\f d_n(t), \phi_n(t))  + \int_0^t\left ( \gamma\| \f q_n \|^2_{\Le} +\lambda_p \left \Vert  j_n\right \Vert _{L^2}^2 + \beta_1 \| \syn v \|_{\Le}^2 +  \beta_2\Vert \dn \cdot \syn v \dn \Vert _{L^2}^2 \right )\text{\emph{d}} s \\  \quad +\int_0^t \left (\beta_3\|\syn v \dn \|^2_{\Le} + \beta_4 \| \an\cdot \syn v \an\|^2_{L^2}  + \beta_5 \| \syn v \an \|^2_{\Le} + \beta_6\| \dn\cdot \syn v\an \|_{L^2}^2\right )\text{\emph{d}}   s
 \\
\leq  \frac{1}{2}\|\f v_n(0)\|_{\Le}^2 +  \F(\f d_n(0), \phi_n(0)) +
 c\int_0^t\| \f g\|_{\Vd}^2 \text{\emph{d}}  s\,
\label{entroin}
\end{multline}
for all $t $ in any compact subinterval of $[0,T_n)$.
\begin{proof}
Starting from equation~\eqref{entro1}, we have to estimate the terms on the right-hand side. Since we assume the strict inequalities~\eqref{con2}, we can find $\zeta \in (0,1)$ such that
\begin{align*}
|2\kappa_3|&\leq  \zeta \sqrt{\alpha_1}\sqrt{\tau_1}\, ,&
|4\kappa_4|& \leq \zeta \sqrt{\alpha_1}\sqrt{2\kappa_2}\, ,\\
|4\kappa_5| & \leq \zeta \sqrt{2\kappa_2}\sqrt{\tau_1}\, ,&
4|\kappa_6-\kappa_1\lambda|&\leq  2 \zeta \sqrt{2\tau_2 -4\kappa_1^2\gamma} \sqrt{2\alpha_5+\lambda/\gamma- \lambda^2/\gamma }\, .
\end{align*}
Every term in the last two lines on the right-hand side of~\eqref{entro1} is estimated by Young's inequality such that
\begin{align}
\begin{split}
& |2 \kappa_3 \|( \dn\cdot \syn v \dn,\an\cdot\syn v\an )|   +| 4 \kappa_4| | (\dn \cdot \syn v \an,\dn \cdot \syn v \dn )| \\
& \quad+| 4 \kappa_5|  | (  \dn \cdot\syn v \an , \an \cdot \syn v \an )|
   + 4 |  \kappa _6 - \kappa_1 \lambda|  |  ( \syn v \dn ,\syn v \an)|
\\  & \leq
   \left   |\zeta (\sqrt{\alpha_1}\dn\cdot  \syn v \dn,\sqrt{\tau_1}\an\cdot\syn v\an )\right | +
\left |\zeta (\sqrt{2\kappa_2} \dn \cdot \syn v \an,\sqrt{\alpha_1}\dn \cdot \syn v \dn )\right |\\
&\quad +\left |\zeta (\sqrt{2\kappa_2}\dn \cdot \syn v \an,\sqrt{\tau_1}\an \cdot \syn v \an)\right |\\
& \quad +2\left |\zeta ( \sqrt{2\alpha_5+\lambda/\gamma- \lambda^2/\gamma }\syn v \dn, \sqrt{2\tau_2 -4\kappa_1^2\gamma} \syn v \an ) \right |\\
& \leq  \zeta\alpha_1\Vert \dn \cdot \syn v \dn \Vert _{L^2}^2 + \zeta 2\kappa_2 \| \dn \cdot\syn v\an \|_{L^2}^2    +\zeta \tau_1\| \an\cdot \syn v \an\|^2_{L^2} \\ & \quad +\zeta (2\tau_2 -4\kappa_1^2\gamma) \| \syn v \an \|^2_{\Le}+\zeta (2\alpha_5+\lambda/\gamma- \lambda^2/\gamma )\|\syn v \dn \|^2_{\Le}
\, .
\end{split}\label{Y1}
\end{align}

As a second step, we estimate the last term in the first line on the right-hand side of~\eqref{entro1}. Therefore, we use the definition of the norm of the dual space~$(\V)^*$ as well as Korn's first inequality~(see McLean~\cite[Theorem 10.1]{mclean})
and again Young's inequality such that
\begin{align}
\langle \f g, \f v_n \rangle &\leq \| \f g\|_{\Vd} \| \f v_n \|_{\V} \leq c\| \f g\|_{\Vd}  \|\syn v\|_{\Le} \leq \frac{c^2}{2\alpha_4} \| \f g\|_{\Vd}^2  + \frac{\alpha_4}{2}\|\syn v\|_{\Le}^2\, .\label{Y2}
\end{align}

Inserting the inequalities~\eqref{Y1} and~\eqref{Y2} into the energy equation~\eqref{entro1} and choosing the constants $\beta_i$ appropriately gives the claimed energy inequality~\eqref{entroin}.
\end{proof}
\end{cor}

All results achieved up to this point are proved for general free energies. We have only assumed differentiability, which is important for the calculation in~\eqref{Fd3}. In the following, we use the specific form of the free energy given in~\eqref{free} and~\eqref{W}.

\begin{lem}[A priori estimate \RM{1}]
\label{cor2}
Let the assumptions of Theorem~\ref{thm:main}  be fulfilled. Then the following a priori estimate holds for the solutions $(\f d_n, \phi_n, \f v _n)$ $(n\in \N)$  to the approximate problem~\eqref{eq:dis}:
\begin{align}
\begin{split}
\frac{1}{2}&\| \f v _n \|_{L^\infty(\f L^2)}^2 +  \sup_{t\in[0,T]}
\mathcal{F}(\f d_n(t), \phi_n(t))
 + \gamma\| \f q_n \|^2_{L^2(\Le)}+ \lambda_p \left \Vert  j_n\right \Vert _{L^2(L^2)}^2+ \beta_1 \| \syn v \|_{L^2(\Le)}^2 +  \beta_2\Vert \dn \cdot \syn v \dn \Vert _{L^2(L^2)}^2  \\ & \quad + \beta_3\|\syn v \dn \|^2_{L^2(\Le)} + \beta_4 \| \an\cdot \syn v \an\|^2_{L^2(L^2)}  + \beta_5 \| \syn v \an \|^2_{L^2(\Le)} + \beta_6\| \dn \cdot\syn v\an \|_{L^2(L^2)}^2
 \\ \leq
& c\left (\|\f v_0\|_{\Le}^2+\| \f g\|_{L^2(\Vd)}^2 + \|\nabla \f d_0 \|_{\Le}^4+ \| \phi _0 \|_{\Hi}^4+  \| \f d _1\|_{\f H^{3/2}(\partial \Omega)} ^4 +1\right ) \,,
\end{split}
\label{entrodiss}
\end{align}
where the positive  constants $\beta_i>0$ ($i \in \{ 1, \ldots , 6\}$) are given in Corollary~\ref{cor1}.
\begin{proof}
Let us show that the terms on the right-hand side of~\eqref{entroin} depending on the initial values
 can be estimated independently of $n$. We recall that $P_n$ is the $\Ha$-orthogonal projection such that $ \| \f v_n (0)\|_{\Le} = \| P_n \f v _0 \|_{ \Le} \leq \| \f v_0\|_{\Le}$.
The required estimate for the free energy is proved in Lemma~\ref{lem:Fdis} such that
\begin{align*}
\F(\dn(0), \phi_n(0))=   \F( \Se \f d_1 + R_n (\f d_0- \Se \f d_1) ,Q_n  \phi_0 ) \leq c ( \|\Se \f d_1 + R_n (\f d_0- \Se \f d_1) \|_{\He}^4+ \|Q_n \phi _0 \|_{\Hi}^4+1)\, .
\end{align*}
Since $R_n$ and $Q_n$ are orthogonal projections (see Section~\ref{sec:dis}), we  observe that
\begin{align*}
 \|\Se \f d_1 + R_n (\f d_0- \Se \f d_1) \|_{\He}& \leq  \|R_n(\f d_0- \Se \f d_1) \|_{\He} + \| \Se \f d_1 \|_{\He} \\& \leq c \|  \f d_0- \Se \f d_1 \|_{\He} + \| \Se \f d_1 \|_{\He} \leq c \| \nabla \f d _0\|_{\Le} + c \| \f d _1\|_{\f H^{3/2}(\partial \Omega)}
 \intertext{as well as}
 \| Q_n \phi_o\|_{\Hi} & \leq c \| \phi_o \|_{ \Hi}\,.
 \end{align*}

This shows that the right-hand side of~\eqref{entroin} can be  estimated from above by a constant that depends on $\f d_0$, $\f d_1$, $ \phi_0$, $\f v_0$, and $ \f g$ but not on $n$.
Therefore, the estimate~\eqref{entroin} holds for all $t\in [0,T_n)$. This finally shows that there is no blow-up for the approximate solution and thus we obtain global-in-time existence of a solution, which more over satisfies~\eqref{entrodiss}.

\end{proof}

\end{lem}

\begin{lem}[A priori estimate \RM{2}]
\label{cor:apri2}
Let the assumptions of Theorem~\ref{thm:main} be fulfilled. Then there exists a constant $c>0$ such that
\begin{align}
\begin{split}
&\| \f v _n \|_{L^\infty(\f L^2)}^2 +   \| \dn\|_{L^{\infty}(\He)}^2 +\|\phi_n\|_{L^\infty(\H)}^2   +  \| \f v_n  \|_{L^2(\V)}^2 +  \Vert \dn \cdot \syn v \dn \Vert _{L^2(L^2)}^2   + \|\syn v \dn \|^2_{L^2(\Le)}\\ & \quad +  \| \an\cdot \syn v \an\|^2_{L^2(L^2)}  +  \| \syn v \an \|^2_{L^2(\Le)} + \| \dn \cdot\syn v\an \|_{L^2(L^2)}^2
+ \|\Delta  \f d_n\|_{L^2(\Le) } ^2 + \|\Delta^2 \phi_n \|_{L^2(L^2)}^2  \leq c
\end{split}\label{apri2}
\end{align}
holds for the solutions $(\f d_n , \phi_n,\f v_n)$ $(n\in\N)$ of the approximate system~\eqref{eq:dis}.
\end{lem}

\begin{proof}
The estimate follows from Lemma~\ref{lem:coerc} due to the structure of the free energy potential.
Indeed,
 inequality~\eqref{Fcoerc} inserted in  estimate~\eqref{entrodiss} implies
\begin{align}
\begin{split}
\frac{1}{2}&\| \f v _n \|_{L^\infty(\f L^2)}^2 +  \eta (\|\nabla \dn \|_{L^\infty(\Le)}^2 + \|\phi_n\|_{L^\infty(\H)}^2) + \gamma\| \f q_n \|^2_{L^2(\Le)}+ \lambda_p \left \Vert  j_n\right \Vert _{L^2(L^2)}^2 \\ & \quad  + \beta_1 \| \syn v \|_{L^2(\Le)}^2 +  \beta_2\Vert \dn \cdot \syn v \dn \Vert _{L^2(L^2)}^2   + \beta_3\|\syn v \dn \|^2_{L^2(\Le)} \\ & \quad+ \beta_4 \| \an\cdot \syn v \an\|^2_{L^2(L^2)}  + \beta_5 \| \syn v \an \|^2_{L^2(\Le)} + \beta_6\| \dn \cdot\syn v\an \|_{L^2(L^2)}^2
 \\ \leq
& c\left (\|\f v_0\|_{\Le}^2+\| \f g\|_{L^2(\Vd)}^2 + \|\nabla \f d_0 \|_{\Le}^4+ \| \phi _0 \|_{\H}^4+ \| \f d _1\|_{\f H^{3/2}(\partial \Omega)}^4 +1\right )
 =:c_1\, .
\end{split}\label{apri1}
\end{align}
The above inequality shows the boundedness of the $L^{\infty}(\He)$-norm of the director (see Corollary~\ref{cor:estim})
and the boundedness of the $L^\infty(\H)$-norm of the layer function. Estimate~\eqref{apri1} allows us to prove bounds for the $L^2$-norm of $\Delta\f d_n$ and $\Delta^2\phi_n$, respectively.
Indeed, with~\eqref{varFd} we observe that---analogously to~\eqref{Rd}---
\begin{align*}
\| \f q_n\|_{ \Le}^2 \geq \frac{1}{2}\| R_n \di ( \f \Lambda : \nabla \dn ) \|_{\Le}^2 - \| R_n R_{\dn}\|_{ \Le}^2 \,.
\end{align*}
Since $R_n$ is the $\Le$-orthogonal projection onto $Y_n$, which is spanned by the eigenfunctions to the operator defined in~\eqref{boundaryvalueproblem}, we find that
\begin{align*}
\| \f q_n \|_{ \Le}^2 \geq \frac{1}{2}\|  \di ( \f \Lambda : \nabla \dn ) \|_{\Le}^2 - \|  R_{\dn}\|_{ \Le}^2 \,,
\end{align*}
and we can follow the same argumentation as in the proof of Lemma~\ref{lem:coerc} obtaining an estimate analogously to~\eqref{dcoerc}.
Similarly, we observe with~\eqref{varp} that---analogously to~\eqref{jco}---
\begin{align*}
\|j_n \|_{ L^2}^2 \geq{}& \frac{k_5^2}{2} \| Q_n \Delta^2 \phi_n \|_{L^2} ^2 - \| Q_n R_{\phi_n}\|_{L^2}^2
\geq {} \frac{k_5^2}{2} \|  \Delta^2 \phi_n \|_{L^2} ^2 - \|  R_{\phi_n}\|_{L^2}^2\,,
\end{align*}
and we finally obtain an estimate analogous to~\eqref{phicoerc}.

Inserting the coercivity-like estimates for $\f q_n$ and $j_n$ into~\eqref{apri1} and using~\eqref{apri1} again to estimate the $L^{\infty}(\He)$-norm of $\f d_n$ and the $L^\infty ( \H)$-norm of $\phi_n$  shows that
\begin{align*}
\begin{split}
\frac{1}{2}&\| \f v _n \|_{L^\infty(\f L^2)}^2 +  \eta (\|\nabla \dn \|_{L^\infty(\Le)}^2 + \|\phi_n\|_{L^\infty(\H)}^2)
+2\eta\| \Delta \dn \|_{L^2(\Le)}^2
+ \eta\|\Delta^2 \phi_n\|_{ L^2 (L^2)}^2 -\eta\|\Delta  \dn \|_{L^2(\Le)}^2 \\
&\quad + \beta_1 \| \syn v \|_{L^2(\Le)}^2 +  \beta_2\Vert \dn \cdot \syn v \dn \Vert _{L^2(L^2)}^2   + \beta_3\|\syn v \dn \|^2_{L^2(\Le)} \\ & \quad+ \beta_4 \| \an\cdot \syn v \an\|^2_{L^2(L^2)} + \beta_5 \| \syn v \an \|^2_{L^2(\Le)} + \beta_6\| \dn \cdot\syn v\an \|_{L^2(L^2)}^2
 \\ \leq
& c_1 +  c (\|\nabla  \dn \| _{L^\infty(\Le)}^{78}+ \| \phi_n\|_{L^\infty(\H)}^{42}+\|\f    d_1 \| _{\f H^{3/2}(\partial\Omega)}^{78}+1) \leq c_1 + 2 c \left ( {c_1}+1 \right )^{39} \, .
\end{split}
\end{align*}
This finally proves the assertion.

\end{proof}

\begin{lem}\label{dtn}
Let the assumptions of Theorem~\ref{thm:main} be fulfilled. Then there exists a constant $C>0$ such that
 the time derivatives of the solutions $( \f d_n, \phi_n , \f v_n)$ $(n\in\N)$ to the approximate system~\eqref{eq:dis}   obey the estimate
\begin{align}
 \| \partial_t \f d_n\|_{L^{4/3}(\Le)} + \| \t \phi_n \|_{L^2(\Le)} +  \| \partial_t \f v_n\|_{L^{2}((\f H^2 \cap \V)^*)} \le C\, , \quad  \,. \label{timeabs}
\end{align}
\end{lem}
\begin{proof}
The first goal is to estimate the time derivative of the solution to the approximate director equation~\eqref{ddis}.
We test $\partial_t \f d_n$ with an arbitrary function $\f \psi \in \Le$ in the $\Le$-inner product. Since $\partial_t\f d_n(t) \in Y_n$, we can insert the projection $R_n$ such that
\begin{align}
\begin{split}
(\partial_t \f d_n, \f \psi) = ( \partial_t \f d_n , R_n \f \psi )
&= -( (\f v_n\cdot \nabla )\f d_n -\skn{v}\f d_n , R_n \f \psi) -( \lambda \syn v \f d_n+2\kappa_1\gamma \syn v \an + \gamma \f q_n , R_n \f \psi )\,.
\end{split}\label{dnters}
\end{align}
Inserting equation~\eqref{ddis} is only allowed due to the application of the projection~$R_n$.
The time derivative in the $\Le$-norm is estimated with the definition of the dual norm such that
\begin{align}
\begin{split}
\sup_{{\|\f \psi\|_{\Le}\leq 1}} | ( \partial_t \f d_n , \f \psi)| \leq{}&  \sup_{{\|\f \psi\|_{\Le}\leq 1}}\left ( \| ( \f v_n \cdot \nabla )\f d_n-\skn{v}\f d_n \|_{\Le}  \right ) \| R_n \f \psi\|_{\Le} \\
 &+\sup_{{\|\f \psi\|_{\Le}\leq 1}}\left ( \lambda\|\syn{v} \f d_n \|_{\Le} + 2\kappa_1\gamma\|\syn v \an\|_{\Le} + \gamma\| \f q_n \|_{\Le} \right ) \| R_n \f \psi\|_{\Le} \, .
\end{split}\label{dntabs}
\end{align}
We recall that~$\|R_n \f  \psi \|_{\Le} \leq \| \f \psi\|_{\Le}$.
Additionally, the boundedness of $\|\syn{v}\f d_n\|_{L^2(\Le)} $, $\| \syn{v}\an\|_{ L^2 ( \Le ) }$, and $\| \f q_n\|_{L^2(\Le)} $ is granted by the a priori estimate~\eqref{entrodiss}. What remains is to estimate the first term on the right-hand side of~\eqref{dntabs}. H\"older's inequality is used to estimate the time derivative in the $L^{4/3}(\Le)$-norm:
\begin{align*}
\| \t \f d_n\|_{L^{4/3}(\Le)}
&\leq  \| ( \f v_n \cdot \nabla ) \f d_n\| _{L^{4/3}(\Le)} + \|\skn{v} \f d_n \|_{ L^{4/3}( \Le )}
\\& \quad +\lambda \|  \syn{v}  \f d_n \|_{ L^{4/3} ( \Le ) }+2\kappa_1\gamma\| \syn{v}\an\|_{ L^{4/3} ( \Le ) }+ \gamma \|\f q_n \|_{L^{4/3} ( \Le )}\\
&\leq  \|  \f v_n\|_{L^2(\f L^6)} \|  \f d_n\| _{L^{4}(\f W^{1,3})}
 + \|\f  v_n\|_{L^2( \Hb )}\| \f d_n \|_{ L^{4}( \f L^{\infty} )}\\& \quad +c \left ( \|  \syn{v} \f d_n \|_{L^2( \Le) }+\| \syn{v}\an\|_{ L^2 ( \Le ) }+  \|\f q_n \|_{L^2(\Le)}\right )\, .
\end{align*}
The appearing norms of $\f d_n$ are bounded in view of the a priori estimate~\eqref{apri2} since
\begin{align*}
\| \dn \|_{ L^4(\f W^{1,3}) } \leq c \| \dn \|_{L^2(\Hc)}^{1/2} \| \dn \|_{L^\infty(\He)}^{1/2} \quad \text{and}\quad \| \dn \|_{ L^4(\f L^\infty) } \leq c \| \dn \|_{L^2(\Hc)}^{1/2} \| \dn \|_{L^\infty(\He)}^{1/2} \,.
\end{align*}

For the time derivative of the approximate layer function, we follow the same reasoning as for the director equation.
Recall that $Q_n$ is the $L^2$-orthogonal projection onto $Z_n$.
In order to estimate the time derivative of $\phi_n$, we insert the projection onto the appropriate subspace. This allows us to use the approximate equation~\eqref{phidis} and estimate further on with H\"older's inequality:
\begin{align*}
\sup_{{\|\zeta\|_{L^2(L^2)}\leq 1}} \left | \intte{( \t \phi_n , \zeta)}\right |
& = \sup_{{\|\zeta\|_{L^2(L^2)}\leq 1}} \left | \intte{ ( -\con\phi_n-\lambda_p j_n , Q_n\zeta)}\right |  \leq \| \f v _n \|_{L^2 (\f L^6)} \|\nabla \phi_n\|_{L^\infty(L^3)} + \lambda_p\|j_n\|_{L^2(L^2)}\, .
\end{align*}
 Thus, the time derivative is bounded in the Hilbert space $L^2(L^2)$.

Recall that $P_n$ is the $(\f H^2 \cap \V)$-orthogonal projection onto $W_n$. Using \eqref{vdis}, it follows for
$ \f\varphi \in \f H^2 \cap \V$ that
\begin{align*}
|\langle \t\f v_n , \f\varphi  \rangle|
={}&  \left|
\langle \f g , P_n\f\varphi \rangle
-
\left(
 ( \f v_n \cdot \nabla ) \f v_n
- \nabla \f d_n^T \f q _n - \nabla \phi_n j_n , P_n\f\varphi\right)
- \left( \f T^V_n ; \nabla P_n \f\varphi \right)
\right|
\\
\le{}& \|\f g\|_{\Vd} \|P_n \f\varphi \|_{\V}
+
\|( \f v_n \cdot \nabla ) \f v_n\|_{\f L^{1}}\|P_n \f \varphi \|_{\f L^{\infty}}
+ \|\nabla \f d_n^T \f q _n \|_{\f L^{1}}\|P_n \f \varphi \|_{\f L^{\infty}}
\\&+ \|\nabla \phi_n j_n \|_{\f L^{3/2}}\|P_n \f \varphi \|_{\f L^{3}}
+ \| \f T^V_n \|_{\f L^{6/5}} \|\nabla P_n \f \varphi \|_{\f L^{6}} \, .
\end{align*}
Since $\f H^2\cap \V $ is continuously embedded in $\V$, $\f L^{\infty}$, $ \f L^{3}$,
and $\f W^{1,6}$, we obtain
\begin{align*}
\|\t\f v_n\|_{(\f H^2 \cap \V)^*} \le {}&c \left(
\|\f g\|_{\Vd} + \|( \f v_n \cdot \nabla ) \f v_n\|_{\f L^{1}}
+ \|\nabla \f d_n^T \f q _n \|_{\f L^{1}}+ \|\nabla \phi_n j_n \|_{\f L^{3/2}}
+ \| \f T^V_n \|_{\f L^{6/5}}  \right)
\end{align*}
and thus
\begin{align*}
\|\t\f v_n\|_{L^{2}((\f H^2 \cap \V)^*)} \le{}& c \left(
\|\f g\|_{L^{2}(\Vd)}
+ \|( \f v_n \cdot \nabla ) \f v_n\|_{L^{2}(\f L^{1})}
\right.
 + \left. \|\nabla \f d_n^T \f q _n \|_{L^{2}(\f L^{1})}
+  \|\nabla \phi_n j _n \|_{L^{2}(\f L^{3/2})}
+ \| \f T^V_n \|_{L^{2}(\f L^{6/5})}  \right).
\end{align*}
With H\"{o}lder's inequality, we observe that
\begin{align*}
\|( \f v_n \cdot \nabla ) \f v_n\|_{L^{2}(\f L^{1})}
\leq{}& \|\f v_n\|_{L^\infty(\f L^{2})} \|\nabla \f v_n\|_{L^2(\f L^{2})}\quad
\text{and}\quad
\left \lVert \nabla \f d_n ^T\f q_n\right \rVert _{ L^{2}(\f L^{1})} \leq {} \left \lVert \nabla \f d_n \right \rVert _{ L^\infty(\f L^ {2})} \left \lVert \f q_n \right \rVert _{ L^{ 2}(\Le )}
\,
\intertext{as well as }
\|\nabla \phi_n j _n \|_{L^{2}(\f L^{3/2})} \leq{}&  \| \nabla \phi _n \|_{L^\infty(\f L^6)}  \| j _n \|_{L^2(L^2)}\,.
\end{align*}
In view of~\eqref{entrodiss} and \eqref{apri2}, the terms on the right-hand sides of the foregoing estimates are bounded.

Finally, we observe with \eqref{lesliedis} and again with H\"{o}lder's inequality that
\begin{align*}
 \| \f T^V_n \|_{L^{2}(\f L^{6/5})}
 \le{} c \Big(&
 \left (\| \f d_n \cdot  \syn v \f d_n \|_{L^2( L^2)}+ \| \f a_n \cdot  \syn v \f d_n \|_{L^2( L^2)}\right ) \left (  \| \f d_n\|_{L^{\infty}(\f L^{6})}^2 + \| \f a_n\|_{L^{\infty}(\f L^{6})}^2  \right )
 \\& +  \| \f a_n \cdot  \syn v \f a_n \|_{L^2( L^2)} \left (  \| \f d_n\|_{L^{\infty}(\f L^{6})}^2 + \| \f a_n\|_{L^{\infty}(\f L^{6})}^2  \right )
 + \|\nabla \f v_n\|_{L^{{2}}( \f L^{2})}\\&  +  \left (  \| \f d_n\|_{L^{\infty}(\f L^{6})} + \| \f a_n\|_{L^{\infty}(\f L^{6})}  \right )\left ( \| \f q_n \|_{L^2(\Le)}
+ \| \syn v \f d_n \|_{L^2(\Le)}+ \| \syn v \f a_n \|_{L^2(\Le)}\right )
\Big)\,,
\end{align*}
which proves the assertion because of  \eqref{apri2} and standard embeddings.
\end{proof}


\subsection{Convergence of the approximate solutions\label{sec:conv}}
\noindent
The a priori estimates~\eqref{entrodiss} and~\eqref{apri2} prove the boundedness of the sequences of solutions to the  approximate problem~\eqref{eq:dis} in different norms.
The Banach--Alaoglu--Bourbaki theorem~\cite[Thm~3.16 on p.\,66]{brezisbook} allows us to deduce relative weak and weak$^*$ compactness of the sequence in the considered spaces. In the following, we are not going to relabel the subsequences.

\begin{lem} \label{lem:wkonv}Let the assumptions of Theorem~\ref{thm:main} be fulfilled. Then there exists a subsequence of the sequence of solutions to the approximate problem~\eqref{eq:dis} and $ \f d$, $\phi$, $\f v$ satisfying~\eqref{weakreg} as well as $\ov{\f q}\in L^2(0,T;\Le)$, $ \ov{j}\in L^2(0,T;L^2)$ such that
 the  convergences
\begin{subequations}\label{wkonv}
\begin{align}
\f d_{n }&\stackrel{*}{\rightharpoonup}  		\f d \quad &&\text{ in } L^{\infty} (0,T;\He)\cap  L^{2} (0,T;\f H^2)\cap W^{1,4/3}(0,T; \Le )\,,\label{w:d}\\
\phi_n &\stackrel{*}{\rightharpoonup}  			 \phi \quad&& \text{ in } L^{\infty} (0,T;\H)\cap L^{2} (0,T;\Hfi)\cap W^{1,2}(0,T;L^2)\,,\label{w:phi}
\\
   \f v_{n }&\stackrel{*}{\rightharpoonup}  	\f v \quad&& \text{ in } L^{\infty} (0,T;\Ha)\cap L^{2} (0,T;\V)\cap W^{1,2}(0,T; (\f H^2\cap \V )^*)\,,\label{w:v}\\
\f q_n &\rightharpoonup  						\ov{\f q} \quad&& \text{ in }  L^{2} (0,T;\Le)\,,\label{w:E}\\
j_n & \rightharpoonup 							\ov{j} \quad&& \text{ in } L^2(0,T;L^2)\, , \label{w:j}
\\
\f d_n\cdot \syn v \f d_n &\rightharpoonup  		\f d\cdot \sy v \f d  \quad &&\text{ in }  L^{2} (0,T;L^2)\,,\label{w:dDd}\\
\f a_n\cdot\syn v \f a_n &\rightharpoonup  		 \f a\cdot\sy v \f a \quad&& \text{ in }  L^{2} (0,T;L^2)\,,\label{w:aDa}\\
\f d_n\cdot \syn v \f a_n &\rightharpoonup  		\f d\cdot \sy v \f a\quad&& \text{ in }  L^{2} (0,T;L^2)\,,\label{w:dDa}\\
\syn v \f d_n &\rightharpoonup  				\sy v \f d \quad&& \text{ in }  L^{2} (0,T;\Le)\,,\label{w:Dd}\\
\syn v \f a_n &\rightharpoonup  				\sy v \f a \quad&& \text{ in }  L^{2} (0,T;\Le)\,,\label{w:Da}\\
\f d_n &\rightarrow \f d \quad &&\text{ in } L^2 ( 0,T; \He)\cap L^{16/5}(0,T; \f W^{1,{16/5}})\cap L^{48/5}(0,T;\f L^{48/5})\, , \label{s:d}
\\
\phi_n&  \rightarrow \phi \quad &&\text{ in } L^2 ( 0, T; \H)\cap L^{24/7}(0,T;W^{2,{24/7}})\cap L^{12}(0,T; \f W^{1,{12}}) \, ,\label{s:phi}
\\
\f v _n & \rightarrow \f v \quad &&\text{ in } L^2(0, T; \Ha)\,\label{s:v}
\end{align}
hold for $n\ra\infty$.
\end{subequations}
\begin{proof}
The  a priori estimates~\eqref{apri2} and \eqref{timeabs}  yield the  weak and weak$^*$  convergences~\eqref{w:d}-\eqref{w:Da}.
The Lemma of Lions--Aubin (Lions~\cite[Th\'eor\`eme 1.5.2]{lions}) ensures the following compact embeddings
\begin{align*}
L^{2} (0,T;\Hc) \cap W^{1,4/3}  (0,T;  \f  L^{2} ) \stackrel{c}{\hookrightarrow} L^{2} (0,T;\He)\, ,\\
L^{2} (0,T;\Hfi) \cap W^{1,2}  (0,T;   \f  L^{2} ) \stackrel{c}{\hookrightarrow} L^{2} (0,T;\H)\, ,
\\
L^2(0,T; \V)\cap W^{1,2}(0,T;(\f H^2\cap \V)^*) \stackrel{c}{\hookrightarrow} L^{2} (0,T;\Ha)\,.
\end{align*}
The convergences of the director~\eqref{w:d}, the layer function~\eqref{w:phi}, the velocity field~\eqref{w:v} as well as
their time derivatives~(\ref{w:d}--\ref{w:v}),
 immediately give the strong convergences with respect to the first space indicated in~(\ref{s:d})--(\ref{s:phi}) as well as~\eqref{s:v}.
With  Lemma~\ref{lem:nir} and Corollary~\ref{cor:phireg}, we observe that
\begin{align*}
\| \f d_n \|_{L^{10/3}(\f W^{1,10/3})} \leq{}& c \| \f d _n \| _{L^2(\Hc)}^{3/5} \| \f d_n \|_{L^\infty(\He)}^{2/5} \text{ and}& \| \f d_n \|_{L^{10}(\f L^{10})} \leq{}& c \| \f d _n \| _{L^2(\Hc)}^{1/5} \| \f d_n \|_{L^\infty(\He)}^{4/5}
\intertext{as well as}
\| \phi_n \|_{L^{14/3}(W^{2,{14/3}})}\leq{}& \| \phi_n\|_{L^2(\Hfi)}^{3/7} \| \phi_n \|_{L^\infty(\H)}^{4/7} \text{ and}& \| \phi_n \|_{L^{14}(W^{1,{14}})}\leq{}& \| \phi_n\|_{L^2(\Hfi)}^{1/7}\| \phi_n \|_{L^\infty(\H)}^{6/7}
\end{align*}
 and thus the boundedness of the sequence $\{\f d_n\}$ in $ L^{10/3}(0,T;\f W^{1,10/3})\cap L^{10} (0,T;\f L^{10})$ and of the sequence $\{\phi_n\}$ in $ L^{14/3}(0,T; W^{2,14/3})\cap L^{14}(0,T;W^{1,14})$. Since
$16/5< 10/3 $, $ 48/5<10$, $24/7< 14/3$, and $12<14$, a standard interpolation argument grants the strong convergence in the last two spaces of~\eqref{s:d} and~\eqref{s:phi}, respectively.
 These strong convergences allow us to identify the limits in~\eqref{w:dDd}-\eqref{w:Da} (recalling that $\an = \nabla \phi_n$).

\end{proof}
\end{lem}


\begin{rem}
The initial values for the approximate equations are defined via the associated orthogonal projections of the given initial datum, i.e.,\,$R_n\f d _0$,  $Q_n\phi_0$, and $P_n\f v_0$, respectively. This ensures that the initial values of the approximate solutions converge strongly to the given initial value, 
\begin{align}
 \f d_n(0) \ra \f d(0) \quad \text{in } \He ,   \qquad \phi_n(0)\ra \phi(0)  \quad \text{in } \H \text{ and} \qquad \f v_n(0) \ra \f v_0  \quad \text{in } \Ha \, .
\end{align}

\end{rem}
%
%
%
The next lemma  identifies the weak limits for the variational derivatives~\eqref{w:E} and~\eqref{w:j}.
\begin{lem}\label{lem:var}
The variational derivatives $\f q_n$ and $j_n$ of the solution to the approximate system  converge weakly to the variational derivative $\f q$ and $j$ of the limit functions given by~\eqref{vari} with $\f d$ and $\phi$ given by Lemma~\ref{lem:wkonv}, i.e.\,
\begin{align}
\f q_n \rightharpoonup \f q \quad \text{in } L^2{(0,T; \Le)} \quad \text{and}\quad j_n \rightharpoonup j \quad \text{in } L^2{(0,T; L^2)} \quad\text{as }n\ra \infty  \, .\label{varcon}
\end{align}
\end{lem}
\begin{proof}
With a priori estimate~\eqref{entrodiss}, we have already deduced the weak convergences~\eqref{w:E} and~\eqref{w:j}. It remains to identify the limits~$\ov{\f q}$ and~$\ov{j}$ in dependence of $\f d$ and $\phi$.
In regard of the composition of the variational derivative~$\f q$ (see~\eqref{varFd}), the higher order term, i.e.,\,$ \nabla ^2 \f d$, occurs only linearly.
For $\f \psi \in L^2(0,T; \Le)$, consider $\f q_n$ tested with $\f \psi$,
\begin{multline}
\intt{ \f q_n , \f \psi} = - \intt{ k_1 \nabla ( \di \f d_n) - k_3 \curl \curl \f d_n ,R_n \f \psi}
\\
  + \intt{ B_0(  ( |\nabla \phi_n|^2 +  \f d_n \cdot \nabla \phi_n - 2 ) \nabla \phi_n ,R_n \f \psi)  + B_1( ( | \nabla \phi_n|^2 \f d_n -  (\f d_n \cdot \nabla \phi_n )\nabla \phi_n ), R_n\f \psi)+ \frac{1}{\varepsilon_1 }( ( |\f d_n |^2 -1 ) \f d_n, R_n \f \psi)}\, .
  \label{q2}
\end{multline}
In the first line of~\eqref{q2} only linear terms of $\nabla^2 \f d_n$ occur. Due to the weak convergence of $\{\f d_n\}$ in $L^2(0,T; \Hc)$, we can pass to the limit in this terms. The second line of~\eqref{q2} depends only on the lower order terms $\f d_n$ and $\nabla \phi_n$, which converge strongly.

Indeed, due to~\eqref{s:d} and~\eqref{s:phi}, we can extract an almost everywhere converging subsequence such that
\begin{align*}
\f d_n (\f x, t) &\ra \f d( \f x ,t) \quad\text{and}& \nabla \f d_n( \f x, t) &\ra \nabla \f d ( \f x ,t) &\text{ for allmost every }(\f x,t)\in \Omega \times (0,T)\, ,\\
\nabla \phi_n (\f x, t) &\ra \nabla \phi( \f x ,t) \quad\text{and}& \nabla^2 \phi_n( \f x, t) &\ra \nabla^2 \phi( \f x ,t) &\text{ for allmost every }(\f x,t)\in \Omega \times (0,T)\, ,
\end{align*}
 where $ \{\nabla \f d _n\}$ is dominated by a function in $L^{16/5}(0,T;\f L^{16/5})$, $\{ \f d_n\} $ by a function in $L^{48/5}(0,T; \f L^{48/5})$, $\{ \nabla^2\phi_n \}$ by a function in $L^{24/7}(0,T;L^{24/7})$, and $\{ \nabla \phi _n \} $ by a function in $L^{12}(0,T;L^{12})$.

Similarly to the estimate~\eqref{Nummer1} in Lemma~\ref{lem:bound}, we can find a dominating function in $L^2(0,T;L^2)$ for the variational derivative~\eqref{q2} and pass to the limit with Lebesgue's theorem on dominated convergence. Note that we put the projection $R_n$  on the test function $\f \psi$ in~\eqref{q2} and that $R_n\f \psi$ converges strongly to~$ \f \psi$ for all $\f \psi \in \Le$.

In a similar way, we show the limiting behaviour for the sequence $\{j_n\}$. Consider the variational derivative of $\F$ with respect to $\phi$, which is given in equation~\eqref{varp},  tested with $\zeta \in L^2(0,T;L^2)$,
\begin{multline}
\intt{( j_n , \zeta )} = \intt{( k_5 \Delta^2 \phi _n ,Q_n \zeta)} - B_0\intt{ \di (  ( |\nabla \phi_n|^2 +  \f d_n \cdot \nabla \phi_n - 2 ) (2\nabla \phi_n- \f d_n )) , Q_n \zeta }
\\
  - \intt{  B_1 \left (\di ( | \f d_n |^2 \nabla \phi_n - ( \f d_n \cdot \nabla \phi_n ) \f d_n),  Q_n \zeta\right )+ \frac{1}{\varepsilon_2}\left ( \di  (( | \nabla \phi_n|^2 - 1 ) \nabla \phi_n),  Q_n \zeta\right )}\, .\label{testedjn}
\end{multline}
The higher order term $\Delta^2 \phi$ occurs again linearly and thus converges weakly due to~\eqref{w:phi}. The lower order terms in~\eqref{varp} also depend on $\nabla^2 \phi$.
Similarly to the estimate~\eqref{Rphi} in Lemma~\ref{lem:bound}, we can find a dominating function in $L^2(0,T;L^2)$ for the variational derivative in~\eqref{testedjn} and pass to the limit with Lebesgue's theorem on dominated convergence. Note that we put the projection $Q_n$  on the test function $\zeta$ in~\eqref{testedjn} and that $Q_n \zeta$ converges strongly to~$ \zeta $ for all $\zeta  \in L^2$.

\end{proof}

\begin{proof}[Proof of Theorem~\ref{thm:main}]
To prove the main result, it remains to prove that the limit of the subsequence of the sequence of solutions $(\f d_n ,\phi_n , \f v_n)$ to the approximate problem~\eqref{eq:dis} fulfills the weak formulation~\eqref{weak}. The essential tools to show this statement are the different convergence results achieved so far.

We start with the director equation.
The time derivative of the approximate solutions converge weakly due to~\eqref{w:d}. From~\eqref{w:d}, the strong convergence~\eqref{s:d}, and the weak convergence of the velocities~\eqref{w:v}, we find that
\begin{align}
\int_0^T \left ( \t \f d_n  + ( \f v_n   \cdot \nabla ) \f d_n  -  (\nabla \f v _n )_{\skw} , \f \psi \right ) \de t  \ra  \int_0^T \left ( \t \f d  + ( \f v  \cdot \nabla ) \f d  - (\nabla \f v )_{\skw} \f d  , \f \psi \right ) \de t \, \label{rdn}
\end{align}
for $\f \psi \in \C_c^\infty (\Omega\times (0,T);\R^3)$ as $n\ra \infty$.
The other appearing semilinear terms converge due to ~\eqref{w:Dd} and~\eqref{w:Da}:
\begin{align}
\int_0^T \left ( \lambda ( \nabla\f   v_n )_{\sym} \f d _n  + 2 \kappa_1 \gamma ( \nabla\f   v_n )_{\sym} \an  , \f \psi \right ) \de t \ra  \int_0^T \left ( \lambda ( \nabla\f   v )_{\sym} \f d  + 2 \kappa_1 \gamma ( \nabla\f   v )_{\sym} \a   , \f \psi \right ) \de t\,
\end{align}
for $\f \psi \in \C_c^\infty (\Omega\times (0,T);\R^3)$ as $n\ra \infty$.
The variational derivative  $\f q_n$ converges due to~Lemma~\eqref{lem:var}. Thus, we have shown the convergence of every term of~\eqref{ddis} and hence, that the limit fulfills~\eqref{eq:dir}.

Due to the strong convergence of $\nabla \phi_n$ according to~\eqref{s:phi} as well as the weak convergence of the velocity field according to~\eqref{w:v}, the time derivative according to~\eqref{w:phi}, and the variational derivative $j_n$ according to~\eqref{varp}, we can take the limit in every term of the approximate layer equation~\eqref{phidis} and obtain
\begin{align*}
\int_0^T \left ( \t \phi_n  +( \f v_n   \cdot \nabla ) \phi_n  + \lambda_p j_n ,  \zeta \right ) \de t  \ra \int_0^T \left (   \t \phi  + ( \f v   \cdot \nabla ) \phi  + \lambda_p  j  ,  \zeta \right ) \de t\,
\end{align*}
for $\zeta \in \C_c^\infty (\Omega\times (0,T))$ as $n\ra \infty$.

Finally, we show that the limit of the solutions to  the approximate system~\eqref{eq:dis} solves~\eqref{eq:velo}.
The term incorporating the time derivative converges due to~\eqref{w:v}.
With~\eqref{w:v} and~\eqref{s:v}, we see the convergence of the convection term such that
\begin{equation*}
\int_0^T ((\f v_n \cdot \nabla) \f v_n , \f \varphi ) \de t
\to
\int_0^T ((\f v \cdot \nabla) \f v , \f \varphi ) \de t \,
\end{equation*}
for all solenoidal $\f\varphi \in \mathcal{C}_c^\infty( \Omega \times (0,T);\R^3)$ as $n \ra \infty$.

The strong convergences of the director and the layer function, see~\eqref{s:d},~\eqref{s:phi}, as well as the weak convergence of the velocity field and the variational derivative $\f q$, see~\eqref{varcon}, grants the weak convergence of the approximate elastic stress~\eqref{lesliedis} to the rearranged elastic stress, where $\ro d $ in~\eqref{Tv} is replaced using~\eqref{dir}.

Since the equation $\ro d +\lambda \sy v \f d + 2\kappa_1 \gamma \sy v \a+ \gamma \f q=0 $ even holds in $L^2(0,T;\Le)$, we can rearrange the viscous stress and obtain~\eqref{Tv}.
Thus, it holds
\begin{equation*}
\intte{(\f T^V_n  ; \nabla \f \varphi  ) }
\to \intte{(\f T^V  ; \nabla \f \varphi  ) } \,
\end{equation*}
for all solenoidal $\f\varphi \in \mathcal{C}_c^\infty( \Omega \times (0,T);\R^3)$ and as $n \ra \infty$.
The remaining term $\nabla \dn^T\f q_n + \nabla \phi_n j_n $ converges weakly due to the weak convergence of the variational derivatives according to~\eqref{varcon} and the strong convergence of the gradients of the director according to~\eqref{s:d} and the layer function according to~\eqref{s:phi}. 
For  all solenoidal $\f\varphi \in \mathcal{C}_c^\infty( \Omega \times (0,T);\R^3)$, the reformulated elastic stress converges as $n\ra \infty$,
\begin{align}
\int_0^T\left (\nabla \dn ^T  \f q_n  + \nabla \phi_n  j_n , \f \varphi  \right )\de t \rightarrow
 \int_0^T\left (\nabla \f d^T   \f q  + \nabla \phi  j , \f \varphi  \right ) \de t  \,. \label{elaconv}
\end{align}
In the limit, the integration-by-parts formula~\eqref{identi} can be applied again such that the equation~\eqref{eq:velo} is even fulfilled with the original elastic stress tensor~\eqref{ela}.
All in all, we proved that a solution in the sense of Definition~\ref{defi:weak} exists.
\end{proof}
In the next part, we introduce a possible adaptation of the model.

\section{Oseen constraint\label{sec:dicus}}
\subsection{Relaxation of the Oseen constraint}
In the modelling of smectic-A liquid crystals, the Oseen constraint $\curl \f a=0$ is often assumed to hold~(see De Gennes~\cite[Section 7.2.1.8.]{gennes}). The layer normal $\f a$ is thus of gradient structure.  Since the normal of the layers is assumed to be a unit vector, it follows that  $\nabla \phi$ is a unit vector. Stewart~\cite{stewart} asserts (as experiments suggest, see~\cite{delaye}) that this will not be the case in the dynamical theory away from equilibrium.  He suggests to choose $\a$ as $\a= \nabla \phi / |\nabla \phi|$. This is convenient since the normal vector should be a unit vector. In contrast to that, one cannot deduce strong convergence of $\a$ from strong convergence of $\nabla \phi$ due to the lack of continuity of the mapping~$\f y \mapsto \f y/|\f y|$.

Instead, we propose to use a continuously differentiable function $ \varrho_{\varepsilon}$,  where $ \varrho_{\varepsilon} $ approximates the mapping $\f x \mapsto \f x /|\f x|$.
We can
define $\a_\varepsilon $ via
\begin{align}
\a_\varepsilon := \nabla \phi   \varrho_\varepsilon (  \nabla \phi). \label{rea}
\end{align}
For every $\varepsilon>0$, $\f a_\varepsilon$ is continuous in $\nabla \phi$ and, therewith, we may infer the strong convergence of $\f a_{\varepsilon,n} $ from the strong convergence of $\nabla \phi_n$.

The proof in this paper can be extended by replacing every occurrence of $\a$ with $ \f a_\varepsilon$.
As already mentioned, this is fairly easy in Section~\ref{sec:conv}, where the convergence of the approximate solutions is shown. Since~$ \varrho_\varepsilon$ is continuous and we have deduced strong convergence of $\nabla \phi_n$, the strong convergence of $\a_{\varepsilon,n} $ follows immediately. The a priori estimates of Section~\ref{sec:energy} can be proved in the same way as long as the coerciveness (Lemma~\ref{lem:coerc}) and the boundedness (Lemma~\ref{lem:Fdis}) of the free energy and its variational derivatives are provided.
The only difference between the proofs of these lemmata due to the redefinition of $\a_\varepsilon$ in~\eqref{rea} is the variational derivative of $\mathcal{F}$ with respect to $\phi$~\eqref{varp}.
\subsection{Variational derivative of the relaxed free energy}
Consider a potential $F\in \C^1( \R^3\times \R^{3\times 3} \times \R^3 ; \R)$ and let this potential define a free energy
via
\begin{align*}
\mathcal{F}( \phi):= \intet{F( \nabla \phi , \nabla ^2 \phi , \a_{\varepsilon})}
\end{align*}
with $\a _{\varepsilon}$ as defined in~\eqref{rea}.
Then the variational derivative of this functional  can be calculated as
\begin{align*}
\va{\phi}(\phi) &= -\di\pat{F}{\nabla \phi } ( \nabla \phi , \nabla ^2 \phi , \a_{\varepsilon})  +\n^2:  \pat{F}{\nabla ^2 \phi } ( \nabla \phi , \nabla ^2 \phi, \a _{\varepsilon} )  - \di \left ( \left (\pat{\a_{\varepsilon}}{\nabla\phi}\right )^T \pat{F}{\a_{\varepsilon}}\right )\\
&= -\di\pat{F}{\nabla \phi } ( \nabla \phi , \nabla ^2 \phi , \a_{\varepsilon})  +\n^2:  \pat{F}{\nabla ^2 \phi } ( \nabla \phi , \nabla ^2 \phi, \a _{\varepsilon} )  - \di \left ( \left (\varrho_{\varepsilon}(\nabla \phi) I + \nabla \phi \otimes \varrho_{\varepsilon}'(\nabla \phi)\right )  \pat{F}{\a_{\varepsilon}}\right )\, ,
\end{align*}
where $I$ denotes the identity matrix in $\R^{3\times 3}$.
Since $\varrho_{\varepsilon}$ and its first derivative are bounded from above, all  the calculations in the proof of Lemma~\ref{lem:coerc} and Lemma~\ref{lem:Fdis} can be carried out in a similar fashion. 
With this method, it is possible to relax the Oseen constraint and, at the same time, to prevent the vector $\a_\varepsilon$ from becoming degenerate.

The model studied this paper with the layer normal as defined in~\eqref{rea} can be seen as a relaxed model of the one proposed by Stewart~\cite{stewart}. It has similar features and it incorporates especially the possible violation of the Oseen constrain $\curl \f a=0$. 
In virtue of the proof in the article at hand, the global existence of weak solutions to this relaxed Stewart model can be proved.

\addcontentsline{toc}{section}{References}

\small

\end{document}